\RequirePackage{ifthen}
\newboolean{springer}
\setboolean{springer}{false}

\ifthenelse {\boolean{springer}}
{
\documentclass[smallextended]{svjour3}
\smartqed
\usepackage[margin=1.3in]{geometry}
} {
\documentclass[11pt]{article}
\usepackage[margin=1in]{geometry}
}

\usepackage{amsmath}
\usepackage{graphicx}
\usepackage{psfrag}
\usepackage{amssymb}
\usepackage{amsmath}
\usepackage{caption}
\usepackage{subcaption}
\usepackage{verbatim,booktabs}
\usepackage{epsfig,url}
\usepackage{float}
\usepackage{mathrsfs}
\usepackage{tkz-euclide}
\usepackage{enumitem} 

\usepackage[usenames,dvipsnames]{pstricks}
\usepackage{epsfig}
\usepackage{pst-grad} 
\usepackage{pst-plot} 
\usepackage{xspace}

\usepackage{comment}
\usepackage{tikz}
\usepackage{accents}
\usepackage{multirow}
\usepackage{array}
\usepackage{longtable}
\usepackage{pdflscape}
\usepackage[ruled,vlined,linesnumbered,noresetcount]{algorithm2e}
\usepackage{hyperref}

\makeatletter
\let\cl@chapter\undefined
\makeatletter

\usepackage[capitalize,noabbrev]{cleveref}

\ifthenelse {\boolean{springer}}
{

\newenvironment{prf}
{\begin{trivlist} \item[] {\em Proof }}
{$\hfill\qed$ \end{trivlist}}

\newenvironment{prfc}[1][]
{\begin{trivlist} \item[] {\em Proof #1 }}
{$\hfill\qed$ \end{trivlist}}

\newcounter{claim} 
\renewenvironment{claim}[1][]
{\refstepcounter{claim} \begin{trivlist} \item[] {\bf Claim~\theclaim}\space#1 \itshape}
{\end{trivlist}}

\newenvironment{cpf}
{\begin{trivlist} \item[] {\em Proof of claim }}
{$\hfill\diamond$ \end{trivlist}}

\newenvironment{spf}
{\begin{trivlist} \item[] {\em Proof of step }}
{$\hfill\triangleleft$ \end{trivlist}}

\journalname{Mathematical Programming Computation}

\newtheorem{assumption}{Assumption}
} {

\usepackage{amsthm}
\newtheorem{theorem}{Theorem}
\newtheorem{proposition}{Proposition}
\newtheorem{lemma}{Lemma}

\newtheorem{definition}{Definition}
\newtheorem{assumption}{Assumption}

\Crefname{chapter}{Chap.}{Chaps.}
\Crefname{section}{Sect.}{Sects.}
\Crefname{proposition}{Prop.}{Props.}
\Crefname{theorem}{Thm.}{Thms.}
\Crefname{definition}{Defn.}{Defns.}
\Crefname{corollary}{Cor.}{Cors.}
\Crefname{assumption}{Assum.}{Assums.}

}

\def\ie{{i.e.,} }

\newcommand{\orgdiv}[1]{#1}%
\newcommand{\orgname}[1]{#1}%
\newcommand{\orgaddress}[1]{#1}%
\newcommand{\postcode}[1]{#1}%
\newcommand{\city}[1]{#1}%
\newcommand{\country}[1]{#1}%

\newcommand{\fnm}[1]{#1}%
\newcommand{\sur}[1]{#1}%

\newcolumntype{L}{>{$}l<{$}}
\newcolumntype{C}{>{$}c<{$}}
\newcolumntype{R}{>{$}r<{$}}

\newcommand{\bR}{\mathbb{R}}

\newcommand{\cD}{{\mathcal{D}}}

\newcommand{\cE}{{\mathcal E}}

\newcommand{\cP}{{\mathcal P}}

\newcommand{\cF}{{\mathcal F}}

\newcommand{\cV}{{\mathcal V}}

\newcommand{\cEp}{\cE_{\mathrm{p}}}
\newcommand{\cFp}{\cF_{\mathrm{p}}}
\newcommand{\cVp}{\cV_{\mathrm{p}}}
\newcommand{\cEc}{\cE_{\mathrm{c}}}
\newcommand{\cFc}{\cF_{\mathrm{c}}}
\newcommand{\cEI}{{\Pi}}

\newcommand{\cEIp}{\cEI_{\mathrm{p}}}
\newcommand{\cVc}{\cV_{\mathrm{c}}}

\newcommand{\dash}{:}
\newcommand{\tabledefline}[2]{\multicolumn{1}{l}{\rlap{#1\ \dash\ #2}}\\}

\SetKwComment{Comment}{$\triangleright$\ }{}
\SetCommentSty{}
\SetKw{Continue}{continue}
\SetKw{Break}{break}
\SetKw{Return}{return}

\DeclareMathOperator{\conv}{conv}

\DeclareMathOperator{\argmax}{argmax}
\DeclareMathOperator{\ineq}{Ineq}

\newcommand{\cplex}{\texttt{CPLEX}\xspace}
\newcommand{\ccity}{\texttt{City}\xspace}
\newcommand{\kgroup}{\texttt{Kgroup}\xspace}

\newcommand{\random}{\texttt{Random}\xspace}

\newcommand{\esmall}{\texttt{Small}\xspace}
\newcommand{\elarge}{\texttt{Large}\xspace}

\allowdisplaybreaks

\newcommand{\bi}{\begin{list}{$\bullet$}{\setlength{\parsep}{0pt}\setlength{\itemsep}{0pt}}}

\title{Formulations of the continuous set-covering problem on networks: a comparative study}

\ifthenelse {\boolean{springer}}
{
\titlerunning{Formulations of the continuous set-covering problem on networks: a comparative study}
\authorrunning{Xu and D'Ambrosio}

\author{\fnm{Liding} \sur{Xu} \and \fnm{Claudia} \sur{D'Ambrosio}}

\institute{
  \fnm{Liding} \sur{Xu},  \fnm{Claudia} \sur{D'Ambrosio} \at
  \orgdiv{LIX CNRS}, \orgname{\'{E}cole Polytechnique, Institut Polytechnique de Paris}, \orgaddress{\city{Palaiseau}, \postcode{91128}, \country{France}} \\
 \email{lidingxu.ac@gmail.com, dambrosio@lix.polytechnique.fr}
}

}
{
\author{\fnm{Liding} \sur{Xu} \thanks{ \orgdiv{LIX CNRS}, \orgname{\'Ecole Polytechnique, Institut Polytechnique de Paris}, \orgaddress{\city{Palaiseau}, \postcode{91128}, \country{France}}.
             E-mail: {\tt lidingxu.ac@gmail.com, dambrosio@lix.polytechnique.fr}}
             \and
\fnm{Claudia} \sur{D'Ambrosio} \footnotemark[1]
}
}

\date{\today}

\begin{document}

\maketitle

\date{}

\begin{abstract}
We study continuous set covering on networks and propose several new MILP formulations and valid inequalities. In contrast to state-of-the-art formulations, the new formulations only use edges to index installed points, and the formulation sizes are smaller. The covering conditions can be represented as  multivariate piecewise linear concave constraints, which we formulate as  disjunctive systems. We propose three MILP formulations based on indicator constraint, big-M, and disjunctive programming techniques for modeling the disjunctive system. Finally, we give a classification of new and old formulations, and conduct experiments to compare them computationally.

\ifthenelse {\boolean{springer}}
{
} {}
\end{abstract}

\ifthenelse {\boolean{springer}}
{
}{
}

\section{Introduction}

 We consider an undirected connected network $N:=(V,E,\ell)$, where $\ell:E \to \bR_+$ is the edge length function. We denote $\ell_e$ the length $\ell(e)$ of edge $e$ for short. The continuum of edges and vertices of $N$ is denoted with $C(N)$. The distance function $d(\cdot,\cdot)$ defines the distance between any two points in $C(N)$, which coincides with the length of the shortest path in $C(N)$ connecting them.

 Throughout this paper,  let  $\delta$ be a given positive real number denoting the covering radius.  A point $p\in C(N)$ is said to $\delta$-cover (cover for short) $p'\in C(N)$ if $d(p,p')\le \delta$ holds. Note that, symmetrically, if $p$ covers $p'$, then $p'$ covers $p$.  Moreover, a $\delta$-path from $p$ is a path with $p$ as its end of length $\delta$, so $\delta$-paths from $p$ defines  all  coverable  points by it in the network.

 The continuous set-covering problem on $N$ is to find a set of  points in $C(N)$ of minimum cardinality that $\delta$-covers the whole network, and is formally stated next.
 
\begin{definition}\label{def.problem}
The Continuous Set-Covering (CSC) Problem on a network $N$ can be formulated as the following optimization problem:
\begin{equation}
\label{cfl}
	\min\Big\{|\cP|: \: \cP \subseteq C(N)  \textup{ and } \forall p\in C(N), \: \exists p'\in \cP, \, d(p,p')\leq \delta\Big\}.
\end{equation}
A finite set $\cP$ satisfying the condition within \eqref{cfl} is called a cover of $N$, and the cover $\cP^*$ achieving the minimum of \eqref{cfl} is called a minimum (optimal) cover.
\end{definition}
\Cref{fig.covergraph} displays a network and the configuration of its covers by three points $A,B,C$. The parts of edges covered by different points are depicted using different styles. We note that the point $C$ is in the interior of an edge.

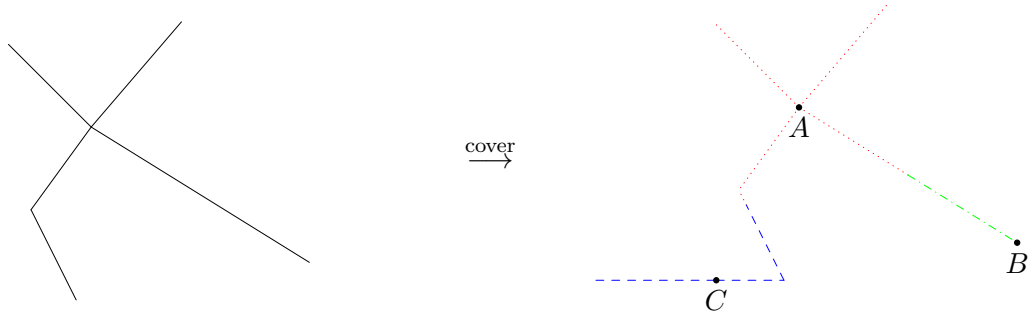
\begin{figure}
\begin{minipage}[h]{0.47\linewidth}
\begin{center}
\begin{tikzpicture}
\tkzDefPoints{0/0/A, 1.2/1.4/B, 2.9/-1.8/C, -1.1/1.1/D, -0.8/-1.1/E,-0.2/-2.3/F,-2.7/-2.3/G}
\tkzDrawSegment(A,B)
\tkzDrawSegment(A,C)
\tkzDrawSegment(A, D)
\tkzDrawSegment(A,E)
\tkzDrawSegment(F,E)
\end{tikzpicture}
\end{center} 
\end{minipage}
\hfill
$ \overset{\text{cover}} {\longrightarrow}$
\vspace{0.2 cm}
\begin{minipage}[h]{0.47\linewidth}
\begin{center}
\begin{tikzpicture}
\tkzDefPoints{0/0/A, 1.2/1.4/B, 2.9/-1.8/C, 1.45/-0.9/CM, -1.1/1.1/D, -0.8/-1.1/E, -0.7/-1.3/EM, -0.2/-2.3/F,-2.7/-2.3/G, -1.1/-2.3/GM}
\tkzDrawSegment[red,dotted](A,B)
\tkzDrawSegment[red,dotted](A,CM)
\tkzDrawSegment[green,dashdotted](CM, C)
\tkzDrawSegment[red,dotted](A,D)
\tkzDrawSegment[red,dotted](A,E)
\tkzDrawSegment[red,dotted](EM,E)
\tkzDrawSegment[blue,dashed](F,EM)
\tkzDrawSegment[blue,dashed](G,GM)
\tkzDrawSegment[blue,dashed](GM,F)
\tkzDrawPoints(A,C,GM)
\tkzLabelPoint[below](GM){$C$}
\tkzLabelPoint[below](C){$B$}
\tkzLabelPoint[below](A){$A$}
\end{tikzpicture}
\end{center}
\end{minipage}

\caption{A network (left) and a configuration of its continuous set covers (right).}
\label{fig.csc}
\end{figure}

The set $\mathcal{P}$ in Definition \ref{def.problem} can represent the locations of ambulance bases  \cite{health}, surveillance cameras \cite{Gusev20}, routing servers in a network of computers \cite{worm}, cranes for construction \cite{crane}, aerial military medical evacuation points \cite{military}, aircraft alert sites for homeland defense \cite{homeland}, or eVTOL safety landing sites in an urban area \cite{liding}. These applications can be modeled as a CSC problem or its variant. The CSC problem is $\mathcal{NP}$-hard  in general  \cite{Hartmann21}.

The only methods that guarantee to find an optimal solution to the CSC problem are based on integer programming.
To the best of our knowledge, the first mixed-integer linear programming (MILP) formulation without any preprocessing technique is proposed  in \cite{Hamacher20}, and it uses edges to index candidates points that must include an optimal cover. In \cite{pelegrin2023continuous}, the authors propose new  formulations using edges and vertices to index candidate points, and study new covering conditions allowing reduction of the sizes of the formulations. Their computational results show that  preprocessing techniques (delimitation and bound tightening) can improve those formulations and lead to state-of-the-art performance.

In this paper, we turn back to the basic edge model in \cite{Hamacher20}. Because it does not need additional variables to model the vertex candidates, it may be useful in practice due to  its simplicity. We show that the preprocessing techniques in \cite{pelegrin2023continuous} can also apply to the basic edge model and result in formulations of smaller sizes. In addition, we study several formulations of the preprocessed edge model and strengthening methods: i) we propose a MILP formulation with indicator constraints; ii) we use big-M and disjunctive programming techniques to construct two MILP reformulations without indicator constraints, iii) we propose valid inequalities to strengthen the MILP reformulations. We also give a classification of the new and existing  formulations. Finally, we conduct computational tests to compare these formulations and give a receipe for selecting formulations.

\subsection{Related work}

There are vast variants and applications of facility location and set cover problems. We review some works related to the CSC problem.  The maximal covering location problem \cite{Church74} is a classical extension of the set cover problem on networks, and  the solution space of its variants could be continuous \cite{chapter-plastria,bansal2017planar}, discrete \cite{cordeau2019benders,chapter-marin}, or on networks \cite{berman2016covering,bucarey2022benders}. In \cite{berman2016covering}, the demand is made over  edges. In \cite{baldomero2022upgrading}, the authors studied the upgrading version of the maximal covering location problem with edge length modifications on networks. Another related problem  is the obnoxious facility location problem \cite{drezner2018weber}, which  aims at locating undesirable facilities that hurt communities. The most common  objective is to maximize the shortest distance to the closest facility, and the problem has various variants featuring multiple facilities on the plane \cite{drezner2019planar,kalczynski2021obnoxious}, $p$-median objective \cite{kalczynski2022obnoxious}, or edge demand on networks \cite{berman2016covering}. We refer to \cite{church2022review} for a recent review on the obnoxious facility location problem. The total edge covering problem \cite{fallah2009covering} aims to determine a set of vertices of a given graph that cover all edges with the extra condition that each edge should be covered only by one vertex. In the following work \cite{sadigh2010mixed}, a variant of this problem allows  partial cover of an edge.  A unified characteristic of facility location problems in a continuous space is studied in \cite{Schobel2019}. We refer to \cite{puerto2018extensive} for a summary of the research progress in facility location problems on networks.

Apart from the integer programming approaches,
discretization methods were  used to tackle the CSC problem. These methods identify subsets of candidate locations that guarantee to contain an optimal solution. 
There are at least three such methods for  variants of the CSC problem \cite{Church79,Hamacher20,Gurevich84},  relying on different assumptions on the network and the covering radius. However, it is difficult to extend discretization methods to solve the CSC for general networks and real radii \cite{pelegrin2023continuous}, because the number of identified candidate locations in these methods may be exponential in the size of the network.

The challenge in this paper stems from modeling multivariate piecewise linear concave constraints. There are some related studies in mixed-integer programming. In \cite{li2009superior}, the authors proposed a big-M approach to model univariate piecewise linear functions. In a more recent work \cite{vielma2018embedding},  new formulations for univariate piecewise linear functions are proposed and compared against existing ones. In \cite{d2010piecewise}, the authors studied formulations of piecewise linear functions for approximating nonlinear bivariate functions. In \cite{rovatti2014optimistic}, the authors studied optimistic MILP modeling of non-linear optimization problems using piecewise linear functions. In \cite{huchette2019combinatorial}, the authors studied bivariate piecewise linear functions through  the lens of disjunctive programming \cite{balas1979disjunctive}. In \cite{huchette2023nonconvex}, the authors propose  mixed-integer programming  formulations for optimization over nonconvex piecewise linear functions.

\subsection{Notation}
We assume that $V$ is totally ordered by the binary relation $\preceq$. Every edge $e \in E$ has a unique representation, $e = (v_a, v_b)$, where $v_a, v_b \in V$, and $v_a \preceq v_b$. We denote $v_a = e(a), v_b = e(b)$. From now on, we take $e=(v_a, v_b)$ indifferently as a continuum in $C(N)$ or as an edge ending at  $v_a, v_b$. For two points $p,p' \in e$, we take $(p,p')$ as the sub continuum with extreme points $p,p'$ in $e$. For an positive integer $k$, let $[k]$ denote the set $\{1,\cdots,k\}$.
\section{Preliminary}

We review several cover conditions and delimitations introduced in \cite{pelegrin2023continuous}.  Then, we define models for the continuous set covering problem.

As in \cite{Hamacher20,pelegrin2023continuous}, we assume that the following assumption holds in the sequel:
\begin{assumption} \label{assumption}
For all $e \in E$, $\delta \ge \ell_e$.
\end{assumption}
This assumption limits the set of tractable networks. In fact, we can make this assumption w.l.o.g. because, by splitting the edges, any network can be converted to a network meeting this assumption without loss of the generality.

In \Cref{def.problem}, we know that the original cover condition is:
\begin{equation}
\label{eq.coverorg}
	\forall p\in C(N), \: \exists p'\in \cP, \, d(p,p')\leq \delta
\end{equation}
 Under \Cref{assumption}, we can alternatively write the following equivalent condition.

 \begin{lemma}\cite{pelegrin2023continuous}
 \label{prop.covercond}
A finite set $\cP$ of points in $C(N)$  is a cover of $N$ if and only if,
for all $e:=(v_a,v_b) \in E$, its  edge cover condition is satisfied:\\ either there exists $p \in \cP\cap e$ or
 \begin{equation}
	\label{eq.coveredge}
    	\sum_{i \in \{a,b\}}   \max_{p \in \cP}  \left\{ \left(\delta - d( v_{i}, p) \right)_+  \right\} \ge \ell_e,
	\end{equation}
where $(x)_+:= \max \{x, 0\}$.
\end{lemma}

This proposition reduces the cover condition over the whole network $N$ to the edge cover conditions over the edges. The argument $\cP$ in the $\max$ function is over all potential points. This may result in a large search space, which can be a set of infinite cardinality without preprocessing.

For sparse networks and small covering radii,  a point in an edge can only cover its neighbors locally (with distance at most $\delta$). Therefore,  several delimitation sets can reduce the search space, \ie the $\argmax$ can shrink from $\cP$ to its subset. We introduce the delimitation used in \cite{pelegrin2023continuous}.

\begin{definition}\label{defa.linksetsv}
For each $v \in V$, the potential covers of $v$ are the candidate point locations to cover $v$:
\begin{align}
 \cE(v)&:= \{e'=(v'_a,v'_b) \in E:\:   d(v,v'_{i}) \le \delta \textup{ for some } i\in \{a,b\} \} \nonumber\\
 \cV(v)&:= \{v' \in V: \:  d(v,v') \le \delta\} \nonumber\\
 \cF(v)&:= \cE(v) \cup \cV(v). \nonumber
\end{align}
\end{definition}

\begin{definition}
\label{def.linksetse}
For each $e=(v_{a}, v_{b}) \in E$, the  complete covers of $e$ are the candidate point locations that can completely cover $e$:
\begin{align}
\cEc(e)& :=\{e' \in E:   \forall p' \in e', \forall  p \in e,\:   d(p,p') \le \delta\}\nonumber\\
   \cVc(e)&:= \{v' \in V:   \forall  p \in e,\: d(p,v') \le \delta \}\nonumber\\
  \cFc(e) &:= \cEc(e) \cup \cVc(e). \nonumber
\end{align}
\end{definition}

If  $\cFc(e)$ is not empty, $e$ is called completely coverable.

\begin{definition} \label{defa.linksetsvincident}
We define the following sets, for each $v\in V$:
\begin{align}
\cEp(v)&:=\{e'\in \cE(v):\: \exists e\in E(v),\: e'\notin\cEc(e)\} \nonumber\\  \cVp(v)&:=\{v'\in \cV(v):\: \exists e\in E(v),\: v'\notin\cVc(e)\}   \nonumber\\
\cFp(v)&:=\cEp(v) \cup \cVp(v), \nonumber
\end{align}
where $E(v):=\{e\in E:\: v\in e\}$ is the set of incident edges to $v$. We call these sets the \emph{partial covers} of $E(v)$.
\end{definition}

For example, we look at delimitation sets of the grid graph  in \Cref{fig.covergraph}, where the edge length is the visual distance between its end nodes in the figure. We note that $\ell_{13} = \ell_{34} < \ell_{14} = \sqrt{2} \ell_{13}  < \ell_{15} = \sqrt{5} \ell_{13}$. Let $\delta = \ell_{14}$. Then, $\cE(1):=\{(1,3), (1,4), (1,5), (3,4)\}, \cV(1):=\{1,3,4\}$; $\cEc((1,3)):=\{(1,3)\}, \cVc((1,3)) = \{1,3\}$; $\cEp(1) := \cE(1), \cVp(1) := \cV(1)$. Change $\delta$ to $\ell_{15}$. Then, $\cE(1):=\{(1,3), (1,4), (1,5), (3,4), (2,4)\}, \cV(1):=\{1,3,4,5\}$; $\cEc((1,3)):=\{(1,3),(3,4)\}, \cVc((1,3)) = \{1,3,4\}$; $\cEp(1) := \cE(1), \cVp(1) := \cV(1)$. Moreover, if we delete $(1,5),(2,4)$, then $\cEp(1) := \varnothing$.

\begin{figure}[!h]
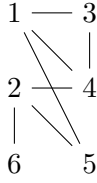

\centering
\usetikzlibrary {graphs.standard}
\tikz
 \graph {
  subgraph I_nm [V={1, 2,6}, W={3,4,5}];
  1 -- { 3, 4, 5 };
  2 -- { 4, 5 };
  3 -- 4;
  6--{2};
  }
;
\caption{A grid graph}
\label{fig.covergraph}
\end{figure}

Using the above delimitation sets, we give a delimited cover condition.

\begin{proposition}
\label{prop.coverdelimitcond}
A finite set $\cP$ of points in $C(N)$  is a cover of $N$ if and only if,
for all $e:=(v_a,v_b) \in E$, the delimited edge cover condition is satisfied:\\ either $\cP \cap \cFc(e) \ne \varnothing$ or
\begin{equation}
\label{eq.coverdlimitedge}
	\sum_{i \in \{a,b\}}  \max_{p \in   \cP \cap \cFp(v_i)} \left\{ \left(\delta - d( v_{i}, p) \right)_+  \right\} \ge \ell_e.
\end{equation}
\end{proposition}
For all $v_i \in V, e \in E$, setting $\cFc(e) = e, \cFp(v_i) = C(N), \cEp(v_i)=C(N)$ (trivial delimitation), the delimited edge cover condition  becomes the edge cover condition \eqref{eq.coveredge}. Thereby, the delimited edge cover condition of \Cref{prop.coverdelimitcond} refines the edge  cover condition of \Cref{prop.covercond}. In the sequel, we  discuss results based on \Cref{eq.coverdlimitedge}.  

We call $\cP$ the set of \textit{candidate points}.
 A \textit{location space} is a subset of $C(N)$ that could  contain a candidate point of $\cP$.
 We next present an alternative definition of  the CSC problem, which changes \Cref{def.problem}  in the following two aspects: i) the location spaces for $\cP$ are restricted to a finite collection $\{C_f \subseteq C(N)\}_{f\in \cF}$, where $\cF$ is a finite index set to be determined later; ii) the covering condition \eqref{eq.coverorg} is replaced by \eqref{eq.coveredge} or \eqref{eq.coverdlimitedge}. Then, the CSC problem can be restated as follows:
 \begin{multline}
\label{cfl.re}
	\min\Big\{|\cP|: \: \cP \subseteq \{p_f:p_f \in C_f\}_{f \in \cF},\, \forall e \in E, \textup{ the edge cover condition \eqref{eq.coveredge} or \eqref{eq.coverdlimitedge} is satisfied} \Big\}.
\end{multline}
A \textit{model} for \eqref{cfl.re} is a mathematical program that computes an optimal cover.

Based on a simple observation, an optimal cover to the CSC problem has at most one candidate  point in each edge. Therefore, there could be two types of candidate points and  associated location spaces. Given an index $f \in \cF$, either  $f = v \in V$, there is one candidate point fixed at $v$, \ie $C_v = \{v\}$; or $f = e \in E$, there is a candidate point in $e$, \ie $C_e = e$. We have the following observation on the index set $\cF$: either $\cF= E$ or $\cF = E \cup V$. This will be later used to classify mathematical formulations for the CSC problem.

For a model with $\cF = E \cup V$, we call it an \textit{edge-vertex model}. Edge-vertex models are extensively studied in \cite{pelegrin2023continuous}. When $\cF = E$, we call it an \textit{edge model}. Edge models only use edges to index candidate points, so one can expect them to have smaller representation sizes. Instead of using full delimitation sets $\cF, \cF_c, \cF_p$, it suffices to restrict  to their subsets $\cE, \cE_c, \cE_p$, which only concern candidate points in edges. Except for the simple model in \cite{Hamacher20}, there are few elaborated edge models and related studies. We propose new edge models in the following.

\section{Delimited edge model}

In this section, we study mathematical programming formulations for the edge model under the delimited edge cover condition (namely, delimited edge model). We first present a mathematical programming formulation with indicator constraints. We then reformulate it into two MILP formulations using big-M and disjunctive programming techniques. Finally, we introduce the valid clique inequalities for these formulations.  

 In the delimited edge model, $\cF = E$, we use  $e'$ to denote edges where points are installed,  and use $e$ to denote   edges to be covered, respectively. We refer to $v_a, v_b$ as the end vertices of $e=(v_a,v_b)\in E$; given $e'=(v'_a,v'_b) \in E$, we refer to $v'_a, v'_b$ as its end vertices.
 
 There are two decisions associated with each  index $e' \in E$. One is to decide whether a candidate point is installed on $e'$. The second is only necessary for those candidate points installed in edges, and consists in determining their locations within the corresponding edges.

To represent the first of the above decisions, we define the following binary variables, which we call  \emph{placement variables}:
\[y_{e'} = \begin{cases}
1 & \textup{ if a point  is installed at } e'\\
0 &\textup{ otherwise}
\end{cases}
\qquad \forall e' \in E.
\]
 To represent the second of the above decisions,  we define the following continuous variables, which we name  \emph{coordinate variables}: $q_{e'}  \in  [0, \ell_{e'}]$. If $y_{e'} = 1$, denote  by $p$ the point installed in $e'$, then $q_{e'}$ should denote the the length of path  connecting $v'_a$ and $p$ in $e'$. Thus, $q_{e'}$ defines the coordinate of $p$ in the network.

With information on complete covers and placement variables, we define the following binary \emph{completely covered variables} to indicate whether an edge $e \in E$ is completely covered:
\[w_e=
\begin{cases}
1 &\textup{ if $\exists e' \in \cEc(e)$ such that } y_{e'} = 1 \\
0 &\textup{ otherwise}.
\end{cases}
\]

For any uncompletely covered edge $e$ (with $w_{e} = 0$), $e$ should be covered by $\delta$-paths from one or two installed points. There are two possible cases. In the first  case, all points in $e$ are covered by a single $\delta$-path passing through $v_a$ or $v_b$ originated from an installed point;  in the second case,  points close to $v_a$ in $e$ are covered by a $\delta$-path passing through $v_a$ originated from an installed point, and points close to $v_b$ in $e$ are covered by a $\delta$-path passing through $v_b$ originated from an installed point. Therefore, we should look at the lengths of the truncated $\delta$-paths within $e$, and these truncated lengths depend on the coordinates of their end points. Then, we define the nonnegative continuous \emph{residual cover variable} $r_v$, for $v \in V$, which satisfies the following constraint:
\begin{equation}
\label{eq.rv}
  r_v
 \le \max_{p \in   \cP \cap \cEp(v)} (\delta - d(v,p)).   
\end{equation}
Note that, for  edges $e$ incident to $v$,   $\max_{p \in   \cP \cap \cEp(v)} (\delta - d(v,p))$ is the maximum  length of truncated $\delta$-paths within $e$ that pass through $v$.

Given uncompletely covered edges, we should compute residual covers of their incident nodes.
We define the following binary variables to model relations between those edges and nodes:
\[x_v=
\begin{cases}
1 &\textup{ if $\forall e \in E(v)$, } w_e = 1 \\
0 &\textup{ otherwise}.
\end{cases}
\]
If $x_v = 0$, then we should compute $r_v$; otherwise, we set $r_v = 0$, because all its incident edges are completely covered, and there is no need to consider its contribution to  covering anymore.

We note that the $\max$ function in \eqref{eq.rv} is convex and piecewise linear over its arguments, so the inequality in \eqref{eq.rv} is a concave multivariate piecewise linear  constraint. Then, we model the constraint using defined variables, and we show that such a constraint is MILP representable after introducing auxiliary binary variables. 

We introduce the following delimitation for all $v \in V$:
\begin{equation}
\label{eq.eipdef}
	\cEIp(v) := \{(e'=(v'_a,v'_b),i') \in \cEp(v) \times \{a,b\}:\: d(v,v'_{i'})\leq \delta\},
\end{equation}
which represents the search space of the argument of the max function in \eqref{eq.rv}. More precisely, we only consider the tuples of edges $e'$ and theirs end nodes $v'_{i'}$, such that we could install a point $p$ in the edge $e'$ in the partial cover $\cEp(v)$, and such a point $p$ could reach $v$ and $v'_{i'}$ through a $\delta$-path. 

Given the delimitation, we can   represent the max function using defined variables. For all $v(e',i') \in   \cEIp(v)$, we define the following affine function with indicators:
\begin{equation}
\label{eq.taudef}
\tau_{ve'i'}: x \mapsto \tau_{ve'i'}(x) := d(v,v_i')+ \mathbf{1}_{i'=a}x +\mathbf{1}_{i'=b}(\ell_{e'} - x).    
\end{equation}
Then, we can express the $\max$ function as the maximum of affine functions \eqref{eq.taudef} applied  on  $q_{e'}$, and the constraint \eqref{eq.rv}  has the following representation:
\[
 r_{v} \le \begin{cases}
 0 & \textup{ if } x_v = 1 \\
 \max_{(e',i') \in \cEIp(v): y_{e'} =1} (\delta - \tau_{ve'i'}(q_{e'})) & \textup{otherwise}.
 \end{cases}
\]

The CSC problem has the following mathematical programming formulation with indicator constraints:
\begin{subequations}
	\label{milpe}
	\begin{align}
   	\min & \sum_{e' \in E} y_{e'} \label{milpe.obj} \\
  	& w_{e} \ge y_{e'}   &  e \in E, e' \in \cEc(e) \label{milpe.completecover1}\\
      	& w_{e} \le  \sum_{e' \in \cEc(e)}y_{e'} 	&  e \in E \label{milpe.completecover2}\\
      	& x_{v}\geq 1-\sum_{e\in E(v)} (1-w_e) &  v\in V \label{milpe.enforcex1}\\
      	& x_{v}\leq w_e &  v\in V, e\in E(v) \label{milpe.enforcex2}\\
  	& \textup{if $x_v = 1$, then $r_v$ = 0; else, $ r_{v} \le \max_{(e',i') \in \cEIp(v): y_{e'} =1} (\delta - \tau_{ve'i'}(q_{e'}))$ } & v \in V \label{milpe.indicator} \\
	&  \ell_e (1 - w_e) \le  r_{v_a} + r_{v_b}	& e \in E \label{milpe.covere}\\
   &   w_e  \in \{0,1\} &	e \in E \label{milpe.varyw}\\
   & x_{v}\in \{0,1\}  & v \in V \label{milpe.varindicator}\\
  	& y_{e'} \in \{0,1\} &  e  \in E \label{milpe.varindicator2}\\
  	&  q_{e'} \in [0, \ell_{e}] &  e  \in E \label{milpe.varcontinuous1}\\
  	& r_{v} \geq 0 & v \in V. \label{milpe.varcontinuous2}
	\end{align}
	\end{subequations}

Given the installation of points indicated by $y_{e'}$, the constraints \eqref{milpe.completecover1} and \eqref{milpe.completecover2} model whether an edge $e$ is completely covered regardless of the concrete locations of these points. Given the complete cover situations indicated by $w_e$, the constraints \eqref{milpe.enforcex1} and  \eqref{milpe.enforcex2} decide whether the residual covers $r_{v_a}$ and  $r_{v_b}$ should be computed, \ie $x_v := \prod_{e \in  E(v)} w_e$ controls the computation of $r_v$. The complex indicator constraints \eqref{milpe.indicator} model the following computation: if $x_v=1$, then $r_v$ is set to zero value; otherwise, $ r_{v} \le  \max_{(e',i') \in \cEIp(v): y_{e'} =1} (\delta - \tau_{ve'i'}(q_{e'}))$ arising in \eqref{eq.coverdlimitedge} is computed. The constraint  \eqref{milpe.covere} stipulates the delimited edge covering condition \eqref{eq.coverdlimitedge} on the edge $e$, if $e$ is not completely covered.

Besides the binary constraints, the mathematical program \eqref{milpe} is highly nonconvex: the indicator constraints \eqref{milpe.indicator} are nonconvex;  their inner conditional constraints $r_v \le \max_{(e',i') \in \cEIp(v):y_{e'} =1} (\delta - \tau_{ve'i'}(q_{e'}))$ are concave constraints. We show that the nonconvexity can be handled by introducing additional binary variables, and the mathematical program \eqref{milpe} is MILP representable.

Before the derivation of the MILP reformulations, we transform the mathematical program \eqref{milpe} into a disjunctive program, whose disjunctions are  concatenated by logical or $\lor$. Then, we obtain its disjunctive reformulation.

\begin{proposition}
The constraint \eqref{milpe.indicator} is equivalent to the following disjunction:
%
\begin{equation}
\label{disjunction}
	\lor_{(e',i') \in \cEIp(v)}
	\left[  \begin{array}{c} x_v = 0 \\
        	y_{e'} =1 \\
        	r_{v} \le   \delta - \tau_{ve'i'}(q_{e'})
        	\end{array}
	\right]
	\lor
	\left[  \begin{array}{c} x_v = 1 \\
        	r_v = 0
        	\end{array}
	\right]
\end{equation}
\end{proposition}

We next derive three MILP reformulations of the disjunctive mathematical program based on indicator, big-M, and disjunctive programming techniques.

\subsection{Indicator formulation}
Many advanced MILP solvers support indicator constraints. This type of constraint allows us to model various implications between two constraints or variables. We will consider indicator constraints consisting of a binary variable and a linear constraint, where the linear constraint should be satisfied when the binary variable takes the value 1. The constraint may or may not hold when the binary variable takes the value 0.

Some solvers, like \cplex, can use special branching strategies to handle  indicator constraints. For many classes of problems, these branching strategies can significantly improve performance.  These solvers can also branch directly on the indicator constraint, or use MILP preprocessing techniques  to automatically transform the formulation into a big-M formulation and  derive  tighter values of big-M  constants that makes the formulation more stable.

We next propose a MILP formulation with indicator constraints. Recall that the disjunctive system \eqref{disjunction} is equivalent to the following system:

\begin{equation}
\label{eq.disjuncsystem}
\begin{aligned}
	& x_{v}  + \sum_{(e', i') \in \cEIp(v)} z_{ve'i'} = 1\\
	&  z_{ve'i'} \le y_{e'} &  (e',i') \in \cEIp(v) 	\\
	&	z_{ve'i'}  \Rightarrow r_{v} \le  \delta - \tau_{ve'i'}(q_{e'})  &  (e',i') \in \cEIp(v)  	\\
 	&	z_{ve'i'} \in \{0,1\} &   (e',i') \in \cEIp(v) \\
 	& x_v \in \{0,1\},
\end{aligned}
\end{equation}
where  $\Rightarrow$ denotes a logical implication. Using this result, we can obtain the  MILP reformulation with indicator constraints (indicator formulation):

\begin{subequations}
\label{milpsos}
	\begin{align}
   	\min & \sum_{e' \in E} y_{e'} \label{milpsos.obj} \\
  	& w_{e} \ge y_{e'}   &  e \in E, e' \in \cEc(e) \label{milpsos.completecover1}\\
      	& w_{e} \le  \sum_{e' \in \cEc(e)}y_{e'} 	&  e \in E \label{milpsos.completecover2}\\
      	& x_{v}\geq 1-\sum_{e\in E(v)} (1-w_e) &  v\in V \label{milpsos.enforcex1}\\
      	& x_{v}\leq w_e &  v\in V, e\in E(v) \label{milpsos.enforcex2}\\
  &  x_{v} + \sum_{(e', i') \in \cEIp(v)} z_{ve'i'} = 1  &   v\in V \label{milpsos.lcoversos} \\
	&   z_{ve'i'} \le y_{e'}  &  v \in V,  (e',i') \in \cEIp(v) \label{milpsos.indicatorei}\\  
 & z_{ve'i'} \Rightarrow  r_{v} \le   \delta - \tau_{ve'i'}(q_{e'}) & v \in V ,  (e',i') \in \cEIp(v) \label{milpsos.coverdist-edge} \\
	&  \ell_e (1 - w_e) \le  r_{v_a} + r_{v_b}	& e \in E \label{milpsos.covere}\\
   &   w_e  \in \{0,1\} &  	e \in E \label{milpsos.varyw}\\
   & x_{v} \in \{0,1\}  & v \in V \label{milpsos.varindicator1}\\
  	& z_{ve'i'} \in \{0,1\}  & v \in V, (e',i') \in \cEIp(v)\label{milpsos.varindicator2}\\
  	& y_{e'} \in \{0,1\} &  e'  \in E \label{milpsos.varindicator3}\\
      	& q_{e'} \in [0, \ell_{e'}] &  e'  \in E \label{milpsos.varcontinuous1}\\
  	& r_{v} \geq 0 & v \in V. \label{milpsos.varcontinuous2}
	\end{align}
\end{subequations}

As modeling of such a formulation is straightforward, one may use this formulation without need of any advanced technique. Thus, we write it here for further comparison with the other advanced formulations. We will examine whether the general-purpose solvers or  our tailored algorithms could handle  the indicator constraints \eqref{milpsos.coverdist-edge} more efficiently.

\subsection{Big-M formulation}

We use big-M techniques to model the disjunction \eqref{disjunction} or the indicator constraints \eqref{milpsos.coverdist-edge}.  For each conjunction term $x_v = 0, y_{e'} = 1,  r_{v} \le   \delta - \tau_{ve'i'}(q_{e'})$,  we introduce additional  binary indicator variable $z_{ve'i'}$ for modelling the disjunction:
\[z_{ve'i'}=
\begin{cases}
1 &\textup{ the conjunction term is activated }\\
0 &\textup{ otherwise}
\end{cases}.
\]

Then, the big-M formulation of the delimited edge model \eqref{milpe} for the CSC problem follows as:
\begin{subequations}
	\label{milpbm}
	\begin{align}
   	\min & \sum_{e' \in E} y_{e'} \label{milpbm.obj} \\
  	& w_{e} \ge y_{e'}   &  e \in E, e' \in \cEc(e) \label{milpbm.completecover1}\\
      	& w_{e} \le  \sum_{e' \in \cEc(e)}y_{e'} 	&  e \in E \label{milpbm.completecover2}\\
      	& x_{v}\geq 1-\sum_{e\in E(v)} (1-w_e) &  v\in V \label{milpbm.enforcex1}\\
      	& x_{v}\leq w_e &  v\in V, e\in E(v) \label{milpbm.enforcex2}\\
  &  x_{v} + \sum_{(e', i') \in \cEIp(v)} z_{ve'i'} = 1  &   v\in V \label{milpbm.lcoversos} \\
	&   z_{ve'i'} \le y_{e'}  &  v \in V,  (e',i') \in \cEIp(v) \label{milpbm.indicatorei}\\  
 &   r_{v} \le M_v (1-x_{v}) &  v\in V \label{milpbm.wbdl}\\
 &   r_{v} \le M_{ve'i'}(1 - z_{ve'i'})  +  \delta - \tau_{ve'i'}(q_{e'}) & v \in V ,  (e',i') \in \cEIp(v) \label{milpbm.coverdist-edge} \\
	&  \ell_e (1 - w_e) \le  r_{v_a} + r_{v_b}	& e \in E \label{milpbm.covere}\\
   &   w_e  \in \{0,1\} &  	e \in E \label{milpbm.varyw}\\
   & x_{v} \in \{0,1\}  & v \in V \label{milpbm.varindicator1}\\
  	& z_{ve'i'} \in \{0,1\}  & v \in V, (e',i') \in \cEIp(v)\label{milpbm.varindicator2}\\
  	& y_{e'} \in \{0,1\} &  e'  \in E \label{milpbm.varindicator3}\\
      	& q_{e'} \in [0, \ell_{e'}] &  e'  \in E \label{milpbm.varcontinuous1}\\
  	& r_{v} \geq 0 & v \in V. \label{milpbm.varcontinuous2}
	\end{align}
	\end{subequations}
 We remark that \eqref{milpbm} is a minor  reformulation of \eqref{milpsos}, where \eqref{milpsos.coverdist-edge} in \eqref{milpsos} is replaced as \eqref{milpbm.wbdl} and \eqref{milpbm.coverdist-edge} in \eqref{milpbm} .
    
One can trivially set all the big-M constants to  $M_v = \delta $ and $M_{vei} =  \delta +  \ell_e$, which are always valid for the formulation. We recall that there also exists big-M formulations for the edge-vertex model \cite{pelegrin2023continuous}. Interested readers may find that the new big-M formulation \eqref{milpbm} for the edge model  is indeed a subsystem of those big-M formulations. This is not a surprise, since the edge model is itself a simplification of the edge-vertex model. Due to this simplicity,  it is easier to derive  disjunctive programming formulation and valid inequalities in the following. We note that the simple model in \cite{Hamacher20} is also an edge model. That model  is exactly the  big-M formulation \eqref{milpbm} without any delimitation. Some advanced techniques in  \cite{pelegrin2023continuous}  can also tighten the big-M constants in the big-M formulation \eqref{milpbm}.  We will use these techniques in our experiments.

\subsection{Disjunctive programming formulation}

 Disjunctive programming is a systematic way to reformulate a disjunctive system into a MILP. This technique does not use any big-M constants but uses additional continuous variables. We will present a MILP reformulation of \eqref{milpe} using the disjunctive programming technique. A general disjunctive programming lemma is as follows.

\begin{lemma} \cite{balas1979disjunctive}
\label{lem.dp}
For $h \in [m]$, let $P_h:=\{x \in \bR^n: A^h x \le b^h, x \ge 0\}$ such that, for all $h \in [m]$, $\{x \in \bR^n:A^hx \le 0, x \ge 0\} = \{0\}$ and let $D = \cup_{h \in [m]}P_h$. Then, 
\begin{multline*}
    D = \\ \{x \in \bR^n:\exists x^1, \dots, x^m \in \bR^n_+\, ,  z \in \{0,1\}^m,\, x= \sum_{h \in [m]}x^h, \, 1= \sum_{h \in [m]}z^h, \, \forall h \in [m], A^h x^h \le b^hz_h\}
\end{multline*} and \begin{multline*}
\conv(D) = \\ \{x \in \bR^n:\exists x^1, \dots, x^m \in \bR^n_+,\,  z \in [0,1]^m,\, x= \sum_{h \in [m]}x^h, \, 1= \sum_{h \in [m]}z^h, \,  \forall h \in [m], A^h x^h \le b^hz_h\}.
\end{multline*}
\end{lemma}
 We call  $A^hx^h \le b^h z_h$ the \textit{duplicated} constraint, $x^h$ the \textit{duplicated} variable of $h$-th conjunction, and $ x= \sum_{h \in [m]}x^h$ the \textit{aggregation} constraint.
 
We find that there are two sufficient conditions to apply the disjunctive programming lemma: for every conjunction: the \emph{nonnegativity condition} $x \ge 0$, and the \emph{zero condition} $\{x \in \bR^n:A^hx \le 0, x \ge 0\} = \{0\}$. These two conditions allow us to apply the disjunctive programming reformulation to the delimited edge model \eqref{milpe}.

We look at the following subsystem of the big-M formulation \eqref{milpbm} associated with a vertex $v \in V$:
\begin{equation}
\label{eq.bigmsystem}
\begin{aligned}
  &  x_{v} +\sum_{(e', i') \in \cEIp(v)} z_{ve'i'} = 1  &   \\
 &   r_{v} \le M_v (1-x_{v}) &  \\
  &	r_{v} \le M_{ve'i'}(1 - z_{ve'i'})  +  \delta - \tau_{ve'i'}(q_{e'}) &  (e',i') \in \cEIp(v) \\
   &  z_{ve'i'} \in \{0,1\}  &   (e',i') \in \cEIp(v) \\
  	& q_{e'} \in [0, \ell_{e'}] &  e' \in \cEp(v) \\
   & x_{v} \in \{0,1\}  &   \\
  	& r_{v} \geq 0. &
\end{aligned}
\end{equation}
We denote by $\cD_v^1$ the feasible set of the above system \eqref{eq.bigmsystem}, and by $\bar{\cD}_v^1$ the feasible set of  the continuous relaxation of \eqref{eq.bigmsystem}. In fact, we find that  $\cD_v^1$ models the following disjunctive system:
\begin{subequations}
\label{eq.disjuncsystem2}
\begin{align}
	& \lor_{(e',i') \in \cEIp(v)} \left[ r_{v} \le   \delta - \tau_{ve'i'}(q_{e'}) \right] \lor \left[ r_v = 0  \right] &  \label{eq.disjuncsystem2.disjunc}  \\
	&   q_{e'} \in [0, \ell_{e'}]  &  e' \in \cEp(v) 	\label{eq.disjuncsystem2.q}   \\
	& r_{v} \geq 0. &	\label{eq.disjuncsystem2.v}
\end{align}  
\end{subequations}

To apply the disjunctive programming lemma,  we intersect every conjunction in \eqref{eq.disjuncsystem2.disjunc} with  constraints \eqref{eq.disjuncsystem2.q} and \eqref{eq.disjuncsystem2.v}.  This results a pure disjunction of several conjunctions:


\begin{equation}
\label{eq.disjuncsystem20}
  \lor_{(e',i') \in \cEIp(v)} \left[  \begin{array}{c}
                                    	0 \le r_{v} \le \delta - \tau_{ve'i'}(q_{e'}) \\
                                    	0 \le q_{e'} \le \ell_{e'}
                                    	\end{array} \right]   
  \lor \left[ \begin{array}{c}
  0 \le r_v \le 0,\\
  \forall e' \in \cEp(v),  0 \le q_{e'} \le \ell_{e'}
   \end{array}\right].
\end{equation}

Now, each conjunction is an inequality system  that contains all relevant variable bounds.

For every $e' \in E$, we introduce the following partial inverse of the delimitation set $\cEIp(v)$ in \eqref{eq.eipdef}:
\begin{equation}
	\cEIp^{-1}(e') := \{(v, i')  :\: \exists i' \in \{a,b\}, (e',i') \in \cEIp(v)\}.
\end{equation}
 For every $(e',i') \in   \cEIp(v)$,  we define the following  affine function, similar to $\tau_{ve'i'}$  in \eqref{eq.taudef}:
\begin{equation}
R_{ve'i'}: (w,y) \mapsto R_{ve'i'}(w,y):= (\delta - d(v,v_i') - \mathbf{1}_{i'=b}\ell_{e'}) y + (\mathbf{1}_{i'=b}  - \mathbf{1}_{i'=a}) w.  
\end{equation}

Using the disjunctive programming \Cref{lem.dp}, we obtain a MILP reformulation of \eqref{eq.bigmsystem}:
\begin{equation}
\label{eq.dpsystem}
\begin{aligned}
  &  x_{v} + \sum_{(e', i') \in \cEIp(v)} z_{ve'i'} = 1  &   \\
&   r_v = \sum_{(e', i') \in \cEIp(v)} r_{ve'i'} &  \\
& q_{e'} =  q_{ve'} + \sum_{i': (e',i') \in \cEIp(v)} q_{ve'i'}   & e' \in \cEp(v)\\
 &   r_{ve'i'} \le R_{ve'i'}(q_{ve'i'}, z_{ve'i'}) &   (e',i') \in \cEIp(v) \\
 &  q_{ve'i'} \le z_{ve'i'} \ell_{e'} & (e',i') \in \cEIp(v) \\
  &  q_{ve'} \le \left(1- \sum_{i':(e', i') \in \cEIp(v)} z_{ve'i'} \right) \ell_{e'} & e' \in \cEp(v) \\
  	& z_{ve'i'} \in \{0,1\} &   (e',i') \in \cEIp(v) \\
  	& r_{ve'i'}, q_{ve'{i'}} \ge 0  &  (e',i') \in \cEIp(v) \\
  	&   q_{e'},  q_{ve'} \ge 0 &  e' \in \cEp(v) \\
   & x_{v} \in \{0,1\}  &   \\
  	& r_{v} \geq 0. &
\end{aligned}
\end{equation}
We note that the disjunctive programming reformulation reuses the binary variables $x_v, z_{ve'i'}$ of the big-M formulation, which indicate the activations of conjunctions.
We denote by $\cD_v^2$ the feasible set of the above MILP system \eqref{eq.dpsystem}, and by $\bar{\cD}_v^2$ the feasible set of  its continuous relaxation.

\begin{theorem}
\label{thm.dp}
Let $v \in V$, $\cD_v^1 = \cD_v^2 \subseteq \conv(\cD_v^2)= \bar{\cD}_v^2 \subseteq \bar{\cD}_v^1$.
\end{theorem}
\begin{proof}
We show that \eqref{eq.dpsystem} is a disjunctive programming formulation of \eqref{eq.disjuncsystem}. The nonnegativity condition is defined as follows: $r_v \ge 0$, and for all $e' \in \cEp(v)$, $q_{e'} \ge 0$. The upper bound condition is defined as follows: for all $e' \in \cEp(v)$, $q_{e'} \le \ell_{e'}$. We first reformulate and classify the conjunctions of \eqref{eq.disjuncsystem} into the following two families of conjunctions:
\begin{enumerate}
	\item $r_v \le 0$, and the nonnegativity and the upper bound conditions hold;
	\item  there exists $(e',i') \in \cEIp(v)$, $r_{v} \le   \delta - \tau_{ve'i'}(q_{e'})$, and the nonnegativity and the upper bound conditions hold.
\end{enumerate}
We find that the nonnegativity condition holds for each conjunction. Then, to check the zero condition of each conjunction, we set constants to zeros. For the first family of conjunctions, setting $\ell_{e'} = 0$, then $r_v = 0$ and $q_{e'} = 0$.  For the second family of conjunctions, expanding $\tau_{ve'i'}(q_{e'}) = d(v,v_i')+ \mathbf{1}_{i'=a}q_{e'} +\mathbf{1}_{i'=b}(\ell_{e'} - q_{e'})$, and setting $d(v,v_i') + \mathbf{1}_{i'=b}\ell_{e'}  = 0$ and $\ell_{e'} = 0$, then $r_v = 0$ and $q_{e'} = 0$. Now we can apply the disjunctive programming  \Cref{lem.dp}, and create duplicated variables of  $r_v,q_{e'}$ for each conjunction.  Regarding the first family of conjunctions, the duplicated variable and constraint of $q_{e'} \le \ell_{e'}$ give rise to that $q_{ve'} \le x_v \ell_{e'}$; for the second family of conjunctions, for all $(e',i') \in \cEIp(v)$, the duplicated variable and constraint of $q_{e'} \le \ell_{e'}$ give rise to that $q_{ve'i'} \le z_{ve'i'} \ell_{e'}$. Since $  x_{v}+ \sum_{(e', i') \in \cEIp(v)} z_{ve'i'} = 1$, the last family of duplicated constraints is (LP) equivalent to  a single constraint $q_{ve'} \le \left(1- \sum_{i':(e', i') \in \cEIp(v)} z_{ve'i'} \right) \ell_{e'}$ with a single duplicated variable $q_{ve'}$ such that the aggregation constraint of $q_{e'}$ becomes $q_{e'} = \sum_{v: (v,i') \in \cEIp^{-1}(e')} q_{ve'i'}$. Finally, we find that \eqref{eq.disjuncsystem} is equivalent to \eqref{eq.dpsystem}, and thus, $\cD_v^1 = \cD_v^2 \subseteq \conv(\cD_v^2)$. By \Cref{lem.dp}, $\conv(\cD_v^2) = \bar{\cD}_v^2$. As $\cD_v^1 \subseteq \bar{\cD}_v^1$ and $\bar{\cD}_v^1$ is a linear program and thus convex, so $ \conv(\cD_v^2) \subseteq \bar{\cD}_v^1$. Then, the result follows.
\end{proof}

Using \Cref{thm.dp}, we obtain a disjunctive programming formulation   of the delimited edge model \eqref{milpe} for the CSC problem:
\begin{subequations}
\label{milpbp}
	\begin{align}
   	\min & \sum_{e' \in \cE} y_{e'} \label{milpbp.obj} \\
  	& w_{e} \ge y_{e'}   &  e \in E, e' \in \cEc(e) \label{milpbp.completecover1}\\
      	& w_{e} \le  \sum_{e' \in \cEc(e)}y_{e'} 	&  e \in E \label{milpbp.completecover2}\\
      	& x_{v}\geq 1-\sum_{e\in E(v)} (1-w_e) &  v\in V \label{milpbp.enforcex1}\\
      	& x_{v}\leq w_e &  v\in V, e\in E(v) \label{milpbp.enforcex2}\\
      	&  x_{v} + \sum_{(e', i') \in \cEIp(v)} z_{ve'i'} = 1  &   v\in V \label{milpdp.lcoversos} \\
	&   z_{ve'i'} \le y_{e'}  &  v \in V,  (e',i') \in \cEIp(v) \label{milpdp.indicatorei}\\  
&   r_v = \sum_{(e', i') \in \cEIp(v)} r_{ve'i'} & v \in V \label{milpdp.dpr}\\
& q_{e'} = q_{ve'} +\sum_{i': (e',i') \in \cEIp(v)} q_{ve'i'}   & v \in V, e' \in \cEp(v) \label{milpdp.dpq}\\
 &   r_{ve'i'} \le R_{ve'i'}(q_{ve'i'}, z_{ve'i'}) & v \in V ,  (e',i') \in \cEIp(v) \label{milpdp.coverdist-edge} \\
  &  q_{ve'i'} \le z_{ve'i'} \ell_{e'} & v \in V, (e',i') \in \cEIp(v) \label{milpbp.qimpliedbd-edge} \\
  &  q_{ve'} \le \left(1- \sum_{i':(e', i') \in \cEIp(v)} z_{ve'i'} \right) \ell_{e'} & v \in V, e' \in \cEp(v)  \label{milpbp.qimpliedbd-node} \\
  	&  \ell_e (1 - w_e) \le  r_{v_a} + r_{v_b}	& e \in E \label{milpbp.covere}\\
   &   w_e  \in \{0,1\} & 	e \in E \label{milpbp.varyw}\\
   & x_{v} \in \{0,1\}  & v \in V \label{milpbp.varindicator1}\\
   &  z_{ve'i'} \in \{0,1\} & v \in V, (e',i') \in \cEIp(v)\label{milpbp.varindicator3}\\
   &  y_{e'} \in \{0,1\} &  e'  \in E \label{milpbp.varindicator2}\\
  	&  q_{e'} \ge 0  &  e'  \in E \label{milpbp.varcontinuous3}\\
 	& r_v \ge 0  & v \in V \label{milpbp.varcontinuous1}\\
 	&  q_{ve'} \geq 0 &   v \in V, e' \in \cEp(v) \label{milpbp.varcontinuous4} \\
 	&  r_{ve'i'}, q_{ve'{i'}} \ge 0  & v \in V, (e',i') \in \cEIp(v)\label{milpbp.varcontinuous2}
	\end{align}
\end{subequations}

The disjucntive programming formulation is same as the big-M formulation, except that the variable $r_v$ are modelled by different constraints. Moreover, the disjunctive programming formulation has a tighter continuous relaxation than the big-M formulation.
\begin{proposition}
The optimal value of the continuous relaxation of \eqref{milpbm} is smaller than  or equal to the optimal value of the continuous relaxation of \eqref{milpbp}.
\end{proposition}
\begin{proof}
Due to \Cref{thm.dp}, the big-M system is a relaxation of the disjunctive programming system, then the result follows.
\end{proof}

Moreover, we can eliminate redundant variables  $ r_{ve'i'} $ and the associated constraints \eqref{milpdp.coverdist-edge} by aggregating these constraints with \eqref{milpdp.coverdist-edge}. 	This results in a set of constraints:
\begin{equation}
 	r_v \le \sum_{(e', i') \in \cEIp(v)} R_{ve'i'}(q_{ve'i'}, z_{ve'i'}), v \in V.
\end{equation}

\section{Subset no good inequalities}
The proposed delimited edge model has three MILP formulations. One may want to add valid inequalities to strengthen these formulations. For MILP formulations of the edge-vertex model, there are simple valid inequalities \cite{pelegrin2023continuous} which concern  edge cover conditions for edges incident to one vertex. Those inequalities involve placement variables associated with vertices, which are not present in the edge model, so the inequalities are unusable here.

We propose new valid inequalities for the edge model that concern edge cover conditions for multiple edges.
We introduce some notation and definitions.
For a subset $A \subseteq E$ of edges, we define $V(A):=\{v \in V: \exists e \in A, \exists i \in \{a,b\}, v = e(i)\}$ the set of vertices which are ends of at least one edge in $A$. We define $m:V(A) \to \cEIp(A)$ a map between $v \in V(A)$ and a partial cover $(e',i') \in \cEIp(v)$, and denote by $M$ all such maps.

A \textit{cover pattern} is a pair of two candidate points that are arguments of the two maximums in \eqref{eq.coverdlimitedge}. 
We define $\cEIp(A):=  \{\{(v,m(v))\}_{v \in V(A)}\}_{m \in M}$ the set of \emph{cover patterns} of $A$. Thereby, each cover pattern $\pi \in \cEIp(A)$ is  a set of pair $(v,(e',i'))$, where index $v \in V(A)$. We define the \emph{edge projector} of the $\cEIp(A)$ the set $\Pr(\cEIp(A)):=\{e \in E: \exists v \in  V(A), m \in M, i' \in \{a,b\}, m(v) = (e,i')\}$, which are edges that can potentially cover edges of $A$.

 Regarding the reformulated CSC problem \eqref{cfl.re} and its formulation \eqref{milpe}, we look at   the edge covering conditions on  $A$.   Reusing variables in formulation \eqref{milpe}, we can represent the feasibility problem (covering condition) for  edges in $A$ as follows:
 
  \begin{subequations}
\label{eq.subsystem}
\begin{align}
  &  \ell_e  \le  r_{v_a} + r_{v_b}	& e \in A  \\
  &  \textup{if } z_{ve'i'} = 1, \textup{then }  r_{v} = \delta - \tau_{ve'i'}(q_{e'}) &  v \in V(A), (e',i') \in \cEIp(v) \\
  & \sum_{(e',i') \in \cEIp(v) } z_{ve'i'}  = 1 & v \in V(A) \\
   &  z_{ve'i'} \in \{0,1\}  &  v \in V(A), (e',i') \in \cEIp(v) \label{eq.subsystem.sos1}\\
  	& q_{e'} \in [0, \ell_{e'}] &  e' \in \Pr(\cEIp(A)) \\
  	& r_{v} \geq 0 & v \in V(A).
	\end{align}
	\end{subequations}

In fact, if we let $A = E$, then the above system is exactly a simplified edge model for the CSC problem without any delimitation of complete covers.
Every cover pattern $\pi  \in \cEIp(A)$ gives a binary assignment $z(\pi)$ of $ z$, such that $\forall (v,e',i') \in \pi$, $z(\pi)_{ve'i'} = 1$. It follows that $z(\pi)$ satisfies the SOS type-1 constraints \eqref{eq.subsystem.sos1}.  We define the following \emph{subset no good inequality} associated with $\pi$:
\begin{equation}
\label{eq.nogood}
	\sum_{(v,e',i') \in \pi} (1- z_{ve'i'})\ge 1.
\end{equation}
We denote the above inequality by $\ineq(\pi)$, and define the \emph{rank} of the inequality as $|\pi|-1$.

With $z$ fixed to $z(\pi)$,  the system \eqref{eq.subsystem} is a linear inequality system, which we call the \emph{feasibility LP problem} associated with $\pi$. For small $|A|$, this problem can be solved very fast. If it is infeasible, then the cover pattern $\pi$ is called \emph{infeasible}. We have the following observation on the validity of subset no good inequality.
\begin{proposition}
	The subset no good inequality  \eqref{eq.nogood} associated with any infeasible cover pattern is valid  for  \eqref{milpe}.
\end{proposition}
\begin{proof}
The subset no good inequality is a valid inequality for the  relaxation \eqref{eq.subsystem}, because it excludes a single infeasible binary assignment. Therefore, the inequality is also valid for the formulation \eqref{milpe} of the original CSC problem.
\end{proof}
Note that this implies that the proposed inequalities are valid for the indicator/big-M/disjunctive programming formulations, as these formulations are reformulations of  \eqref{milpe}.

 In the following,
we  search for infeasible cover patterns.  Two tuples $(v_1, (e'_1,i'_1)),(v_2, (e'_2,i'_2)) \in \cup_{v \in V} \cEIp(v)$  is called \emph{adjacent}, if  $e'_1 = e'_2$ or there exists $e \in E$ that $v_1=e(a), v_2=e(b)$ or $v_1=e(b), v_2=e(a)$. Thus, $ \cup_{v \in V} \cEIp(v)$ form a graph. For a cover pattern $\pi=\{(v, (e',i'))\} \in \cEIp(A)$, we can take it as a subgraph. We call  $\pi$  \emph{connected}, if its underlying subgraph is connected.  We recall that an inequality $az \le b$ dominates another inequality $a'z\le b'$, if every feasible solution to the CSC problem \eqref{milpe} and $az\le b$ is also feasible to $a'z \le b'$. Therefore, we prefer to use non-dominated valid inequalities when strengthening \eqref{milpe}. We have the following dominance relation for inequalities associated with  cover patterns in  $\cEIp(A)$.

\begin{proposition}
\label{prop.dom}
  Let $\pi \in \cEIp(A)$. If $\pi$ is not connected and infeasible, then there exists a connected infeasible subset $\pi' \subseteq \pi$ for which $\ineq(\pi')$  dominates $\ineq(\pi)$. Else, if there exists an infeasible  $\pi' \subsetneq \pi$, then $\ineq(\pi')$ dominates $\ineq(\pi)$.
\end{proposition}
\begin{proof}
Note that, for any subset $\pi'$ of $\pi$,   $\ineq(\pi')$  dominates $\ineq(\pi)$, because $\ineq(\pi')$  excludes more binary points than  $\ineq(\pi)$.
	Assume that $\pi$ is not connected and infeasible. As $\pi$ is infeasible, there must exist a connected subset $\pi'$ that is infeasible. Assume that $\pi$ is  connected and there exists an infeasible subset  $\pi' \subsetneq \pi$, then $\ineq(\pi')$ dominates $\ineq(\pi)$.
\end{proof}

Using the above dominance, when searching for infeasible cover patterns, one can avoid
solving  the feasibility LP associated with some cover patterns. For a nonconnected cover pattern, we do not need to check its feasibility, as there must be some inequalities dominating its inequality when it is infeasible. Moreover, for a connected cover pattern, we do not need to consider it, if some of its subsets are infeasible. The number of  cover patterns grows combinatorially with $|A|$, so one usually searches for cover patterns for small $|A|$. In addition, the induced subset no good inequalities are also sparser for small $|A|$.

We propose an enumeration algorithm \Cref{algo:gen} that generates all at most rank-$K$ subset no good inequalities. The  algorithm exploits \Cref{prop.dom} to find non dominated inequalities, and  it outputs a set of cover patterns, which induce no good inequalities. For the case $k = |A| = 1$, we can explicitly solve the feasibility LP. Finally, the generated inequalities can be added to the proposed MILP formulations.

\begin{algorithm}
\SetAlgoLined
   \textbf{Input:} Network $N= (V, E)$, cover pattern mapping function $\cEIp$, max rank $K$\;
   \textbf{Output:} a set $\Pi$ of cover patterns\;
   Initialize set $\Pi \gets \varnothing$\;
    \For{each rank $k \in [K]$}{
    	Initialize set $\Pi^k \gets \varnothing$\;
    	\For{each subset $A \subseteq E$ of $k$ edges}{
    	\uIf{$k = 1$}{
        	\For{each cover pattern $\pi \in \cEIp(A)$}{
         	Let $\pi = \{(v_1,(e'_1,i'_1)), (v_2,(e'_2,i'_2))\}$\;
         	\If{$\max_{q_{e'_1} \in [0, \ell_{e'_1}]}(\delta - \tau_{v_1e'_1i'_1}(q_{e'_1}),0) + \max_{q_{e'_1} \in [0, \ell_{e'_1}]}(\delta - \tau_{v_2e'_2i'_2}(q_{e'_2}),0) < \ell_{e'}$}
         	{
          	$\Pi^k \gets \Pi^k \cup \{\pi\}$\;
         	}
        	}
    	}
    	\Else{
     	\For{each cover pattern $\pi \in \cEIp(A)$}{
         	\If{$\pi$ is not connected}{
            	\Continue \;
         	}
         	solve the feasibility LP problem associated with $\pi$\;
         	\If{it is infeasible}{
            	$\Pi^k \gets \Pi^k \cup \{\pi\}$\;
         	}
     	}
     	}
 	}
    $\Pi \gets \Pi \cup \Pi^k$
    }
 \Return $\Pi$\;
 \caption{generation of subset no good inequalities}
 \label{algo:gen}
\end{algorithm}

\section{A comprehensive classification of models and formulations}
\label{sec.compare}

We have proposed several new formulations and techniques for the edge model,
and there are also existing formulations for the edge-vertex model \cite{pelegrin2023continuous}. Before we compare these formulations, we need a unified language to distinguish them. In this section, we  classify and compare the proposed models and formulations with the existing ones.

We suggest a method to name formulations by their template models  and employed algorithmic techniques.
 The names of the formulations are abbreviated as ``M-T'', where ``M'' can be  ``EF'' (the formulation of the edge model)/``EVF'' (the formulation of the edge-vertex model)/``LEVF'' (the formulation of the long edge-vertex model), and ``T'' is a combination of labels ``P'',  ``I'', ``D'' and ``V'':  ``P'' means that the formulation is preprocessed by delimitation and bound tightening techniques;  ``I'' means that the formulation is modeled directly using indicator constraints;  ``D'' means  that the formulation is modeled via the disjunctive programming reformulation technique; ``V'' means that the model is added with some valid inequalities, and it has two variants ``V1'' and ``V2'' (which we will explain later).
 
 We remark that,  in  LEVF (long edge-vertex model),   a technique from \cite{pelegrin2023continuous} can directly model the cover condition  on an edge of length greater than $2\delta$. Because this model does not require splitting the long edges in a graph into small edge pieces shorter than $\delta$, this model uses a constant number of variables and constraints for covering each  long edge.

Under this naming system, we list all known MILP formulations for CSC:
\begin{itemize}
  \item EF: the simple big-M formulation for the edge model  \cite{Hamacher20};
  \item EF-P: the big-M formulation for the edge model  preprocessed by delimitation and  bound tightening;
  \item EF-PI:  the indicator formulation of the edge model with  the preprocessing;
  \item EF-PD:  the  disjunctive programming formulation  with  the preprocessing;
  \item EF-PV1: the big-M formulation for the edge model with the preprocessing and  subset no good inequalities for $|A| = 1$;
  \item EF-PV2: the big-M formulation for the edge model with the preprocessing and subset no good inequalities for $|A| \le 2$;
  \item EVF: the simple  big-M formulation for the edge-vertex model  \cite{pelegrin2023continuous};
  \item EVF-P: the big-M formulation for the edge-vertex model with  the preprocessing \cite{pelegrin2023continuous};
  \item LEVF-P: the big-M formulation for the long-edge-vertex model with  the preprocessing  \cite{pelegrin2023continuous}.
\end{itemize}
We remark that the preprocessing procedure is the same for all the models, and the procedure is presented in \cite{pelegrin2023continuous}. The study in  \cite{pelegrin2023continuous} already shows that the preprocessing techniques  can uniformly improve any formulation of the (long) edge-vertex model. Thus, we only report the performances of formulations of the (long) edge-vertex model after preprocessing. We categorize the formulations  in  \Cref{tab.comps}.  The main difference between the formulations in the same row is whether the model contains the vertices as candidate points.

We only derive the disjunctive programming and indicator reformulation for the edge model. In principle, one can also apply these reformulations to the edge-vertex model. However, the formulations may become  more complicated. We remark that we can only derive the long edge formulation for the edge-vertex model.

\begin{table}[]
\begin{tabular}{|l|l|l|}
\hline
          	Techniques                    	& Edge model	& (Long ) edge-vertex model \\ \hline
Big-M  (no preprocessing)                  	& EF        	& EVF           	\\ \hline
Big-M and preprocessing               	& EF-P      	& EVF-P         	\\ \hline
Big-M, preprocessing, and valid inequalities              	&  EF-PV1, EF-PV2 &        	\\ \hline
Disjunctive programming, and preprocessing & EF-PD      	&               	\\ \hline
Indicator constraints, and preprocessing & EF-PI     	&               	\\ \hline
Big-M, preprocessing, and long edge modeling                   	&           	& LEVF-P         	\\ \hline
\end{tabular}
\caption{The classifications of the formulations}\label{tab.comps}
\end{table}

\section{Computational experiments}

In this section, we present the computational experiments comparing the  formulations from \Cref{sec.compare}.

\subsection{Experiment Setup}
We describe the setup of the experiments. The benchmarks, (raw) instance-wise computational results, and source code are publicly released on our project website: \href{https://github.com/lidingxu/cflg/}{https://github.com/lidingxu/cflg/}.

\textbf{Benchmarks.}
We use the same data sets as in \cite{pelegrin2023continuous}. Two data sets come from the literature: \ccity, \kgroup; and one data set is generated synthetically: \random. The \kgroup data set consists of 23 prize-collecting Steiner tree problem instances from \cite{Ljubic2006}. These random geometric instances are designed to have a local structure somewhat similar to street maps. Nodes correspond to random points in the unit square.   The \ccity data set consists of real data of 9  street networks for some German cities, and it was first used in \cite{Kalsics}. The length of each edge is the length of the underlying street segment. The \random data set consists of 24  random network instances generated via Erdős-Rényi binomial method. 

In our test, we group them into two benchmarks of different scales. We put instances with less than 150 edges in a benchmark called \esmall; we put instances with more than 150 edges in a benchmark called \elarge. \Cref{tab.bench} shows some statistics of these two benchmarks.

\begin{table}[]
\centering
\begin{tabular}{|l|r|r|}
\hline
        	Statistics                  	&  \esmall   & \elarge \\ \hline
 Number of instances             	&   32 & 24         	\\ \hline
Min number of edges                  	& 9       	& 185           	\\ \hline
Medium number of edges             	&  69 & 699         	\\ \hline
Max number of edges               	& 148     	&  1035       	\\ \hline
Average number of edges             	&  71 &  584    	\\ \hline
Average graph density           	&  137.5 &   1162.1   	\\ \hline
\end{tabular}
\caption{The statistics of the benchmarks} \label{tab.bench}
\end{table}

\textbf{Covering radii.}
For each network, we define two sets of covering radii: ``Small'' equal to [Average Edge
Length], and ``Large'' equal to ×2 [Average Edge Length], respectively. Each instance is tested using these two different coverage radii. Thus, this doubles the number of test instances from 56 to 112.

\textbf{Development environment.}
The experiments are conducted on a computer with Intel Core i7-6700K  CPU @ 4.00GHZ and 16GB main memory. We use JuMP \cite{DunningHuchetteLubin2017}  to implement our models and interact with MILP solvers. Specifically, we use ILOG \cplex 22 to solve our models. \cplex's parameters are set as their defaults, except that we disable parallelism for each test and set  the MIP absolute gap to 1 (due to the integral objective). Every test runs on a single thread with a time limit of  1800 CPU seconds.

\subsection{Performance metrics}
We describe the performance metrics. Their statistics will be used to evaluate the model performance.

Let $\underline{v}$ be a dual bound and $\overline{v}$ be a primal bound obtained after solving some of the formulations described above,  the relative dual gap  is defined as:
\begin{equation*}
   \sigma := \frac{\overline{v} - \underline{v}}{\overline{v}}.
\end{equation*}
A smaller relative dual gap indicates better primal and dual behavior of the model.

Given a network, its splitted network  has  edge length at most $\delta$.
Let $n_{sd}$ be the number of nodes of the splitted network,  we note that  a trivial primal solution to the CSC problem on this network is the nodes of the splitted network. Therefore, to normalize the primal solution value, we define the relative primal bound  \[v_r := \frac{\overline{v}}{n_{sd}}.\] If $v_r < 1$, then the model finds a solution better than the trivial one.

In order to aggregate performance metrics, we compute their shifted geometric means (SGMs), which provide a
measure for relative differences.  The SGM of values $v_1,...,v_M \geq 0$ with shift $s \geq 0$ is defined
as
\begin{equation*}
  \left(\prod_{i=1}^M (v_i + s)\right)^{1/M} - s.
\end{equation*}

We record the following performance metrics of each instance for each model, and compute their SGMs:

\begin{enumerate}
    \item $t$: the total running time in CPU seconds, with a shifted value of 1 second;
   \item $\sigma$: the relative dual gap, with a shifted value of $1\%$;
      \item $v_r$: the relative primal bound, with a shifted value of  $1\%$.
    
\end{enumerate}

 The running time excludes the time of preprocessing, because we find that the preprocessing usually needs only at most several seconds; for formulations with valid inequalities, the time includes the time to separate valid inequalities. We observe that \cplex may not be able to read and load the data of certain formulation for some instance.  We say that an instance is accepted by a formulation, if the solver can read the formulation of that instance into the machine memory; the instance is solved by this formulation, if an optimal solution could be found within the time limit. When an instance is unaccepted by a formulation, this is usually due to memory issues that the size of the formulation is too large. 
For an unaccepted instance of a formulation, we set its performance metric as  $t= 1800$, $\sigma = 1$ and $v_r = 1$.

\subsection{Analysis of results}

We next look at the computational results of different formulations. First, regarding edge and edge-vertex models, we look at the effects of  preprocessing techniques. Secondly, we compare different models. Then, we compare the proposed formulations for the edge model. Finally, we discuss the choice of  suitable formulations for problems on different scales. Throughout our comparison of existing and new formulations, the readers obtain a practical guide for selecting the most appropriate formulation.

We look at the aggregated results in \Cref{tab.small,tab.large}, where each  table contains the results for the \esmall and \elarge benchmarks, respectively. In each table, the columns list performance metrics under the Small and Large  covering radii, and the rows list performance metrics of the  following formulations: EF, EF-P, EF-PI, EF-PD, EF-PV, EF-PV2, EVF-P, and LEVF-P. Therefore, in each table, for each formulation and covering radius, we also record the S/A metrics, where S denotes the number of solved instances, and A denotes the number of accepted instances. We report the detailed results in \Cref{table:root:detailed} of the Appendix.

\begin{table}[]
\centering
\begin{tabular}{|l|*{4}{r}|*{4}{r}}
\cline{1-9}
\multirow{2}{*}{Formulation} & \multicolumn{4}{l|}{Small radius}   & \multicolumn{4}{l|}{Large radius} 	 \\ \cline{2-9}
                  		 &  \multicolumn{1}{l|}{t (secs)} & \multicolumn{1}{l|}{$\sigma$} & \multicolumn{1}{l|}{$v_r$} & \multicolumn{1}{l|}{S/A} & \multicolumn{1}{l|}{t (secs)} & \multicolumn{1}{l|}{$\sigma$} & \multicolumn{1}{l|}{$v_r$} & \multicolumn{1}{l|}{S/A}  \\ \hline
EF & 455.7 & 41.4\% & 55.1\% & \multicolumn{1}{l|}{8/20} & 223.2 & 40.7\% & 43.7\% & \multicolumn{1}{l|}{12/19} \\ 
EF-P & 269.4 & 21.0\% & 30.3\% & \multicolumn{1}{l|}{11/31} & 21.1 & 18.0\% & 15.4\% & \multicolumn{1}{l|}{28/32} \\ 
EF-PI & 287.6 & 20.2\% & 28.9\% & \multicolumn{1}{l|}{10/32} & 24.1 & 20.3\% & 15.7\% & \multicolumn{1}{l|}{28/32} \\ 
EF-PD & 284.4 & 17.4\% & 30.2\% & \multicolumn{1}{l|}{12/31} & 42.9 & 9.6\% & 15.5\% & \multicolumn{1}{l|}{26/32} \\ 
EF-PV1 & 252.0 & 19.4\% & 28.8\% & \multicolumn{1}{l|}{11/32} & 24.8 & 6.2\% & 16.4\% & \multicolumn{1}{l|}{26/31} \\ 
EF-PV2 & 235.8 & 31.3\% & 43.9\% & \multicolumn{1}{l|}{10/22} & 109.2 & 26.1\% & 41.4\% & \multicolumn{1}{l|}{14/17} \\ 
EVF-P & 343.6 & 24.6\% & 29.5\% & \multicolumn{1}{l|}{10/32} & 37.2 & 17.1\% & 15.5\% & \multicolumn{1}{l|}{27/32} \\ 
LEVF-P & 279.9 & 21.3\% & 29.4\% & \multicolumn{1}{l|}{11/32} & 31.9 & 23.3\% & 15.8\% & \multicolumn{1}{l|}{27/32} \\ 
\cline{1-9}
\end{tabular}
\caption{Results for the \esmall benchmark (32 instances)} \label{tab.small}
\end{table}

\begin{table}[]
\centering
\begin{tabular}{|l|*{4}{r}|*{4}{r}}
\cline{1-9}
\multirow{2}{*}{Formulation} & \multicolumn{4}{l|}{Small radius}   & \multicolumn{4}{l|}{Large radius} 	 \\ \cline{2-9}
                  		 &  \multicolumn{1}{l|}{t (secs)} & \multicolumn{1}{l|}{$\sigma$} & \multicolumn{1}{l|}{$v_r$} & \multicolumn{1}{l|}{S/A} & \multicolumn{1}{l|}{t (secs)} & \multicolumn{1}{l|}{$\sigma$} & \multicolumn{1}{l|}{$v_r$} & \multicolumn{1}{l|}{S/A}  \\ \hline
EF & 1800.0 & 100.0\% & 100.0\% & \multicolumn{1}{l|}{0/0} & 1800.0 & 100.0\% & 100.0\% & \multicolumn{1}{l|}{0/0} \\ 
EF-P & 1800.8 & 57.0\% & 63.3\% & \multicolumn{1}{l|}{0/24} & 1631.5 & 53.4\% & 39.2\% & \multicolumn{1}{l|}{3/24} \\ 
EF-PI & 1801.0 & 61.3\% & 65.4\% & \multicolumn{1}{l|}{0/24} & 1608.0 & 55.2\% & 56.8\% & \multicolumn{1}{l|}{1/24} \\ 
EF-PD & 1800.3 & 67.9\% & 69.5\% & \multicolumn{1}{l|}{0/13} & 1630.8 & 63.3\% & 40.4\% & \multicolumn{1}{l|}{2/13} \\ 
EF-PV1 & 1770.0 & 52.7\% & 60.9\% & \multicolumn{1}{l|}{0/24} & 1718.2 & 59.6\% & 79.7\% & \multicolumn{1}{l|}{0/20} \\ 
EF-PV2 & 1796.2 & 69.5\% & 80.0\% & \multicolumn{1}{l|}{0/7} & 1766.5 & 68.4\% & 64.6\% & \multicolumn{1}{l|}{0/7} \\ 
EVF-P & 1800.7 & 61.8\% & 55.1\% & \multicolumn{1}{l|}{0/24} & 1659.3 & 60.5\% & 73.1\% & \multicolumn{1}{l|}{1/24} \\ 
LEVF-P & 1801.0 & 43.5\% & 46.4\% & \multicolumn{1}{l|}{0/24} & 1660.8 & 51.2\% & 25.6\% & \multicolumn{1}{l|}{2/24} \\ 
\cline{1-9}
\end{tabular}
\caption{Results for the \elarge benchmark (24 instances)} \label{tab.large}
\end{table}

\textbf{Effects of the preprocessing techniques.} In this part, we  compare the EF and EF-P formulations. We find that the EF formulation from \cite{Hamacher20} without any preprocessing performs poorly, as the solver cannot accept many instances. The EF-P formulation has a consistent and significant improvement compared to the EF formulation in terms of the running time, the gap, and finding better  solutions. For example, the EF formulation of the instances in the large benchmark is too large to be read into  \cplex, while \cplex can read and load the EF-P formulations of all instances can be and find improved solutions compared to the trivial solution. Because the preprocessing techniques can  strengthen  formulations and reduce the sizes of formulations, these properties are amenable in practice. Thus, from now on, we will only consider the preprocessed formulations.

\textbf{Comparison of models.}
We compare different models: the edge model,  the edge-vertex model, and the long edge-vertex model. We want to put these models in the same baseline, so that we can understand the effect of  modeling. Thus, we consider only  their big-M formulations that are preprocessed by delimitation and bound tightening techniques: EF-P, EVF-P, and LEVF-P. We first compare the EF-P and EVF-P formulations, which only differ in the space of candidate points. For all the benchmarks and radii, EF-P is better than EVF-P in terms of the running time, the gap, and solving solutions, except that EF-P is slightly worse in improving solutions for the \elarge benchmark under the small radius. This is because the edge model needs fewer candidate points than the edge-vertex model, and the EF-P formulation has fewer variables than the EVF-P formulation. Since we can filter out the EVF-P formulation from our comparison, we next compare the EF-P formulation (of the edge-vertex model) and LEVF-P formulation (of the long edge-vertex model). For all the benchmarks and radii, LEVF-P is better than EF-P in terms of the running time, the gap, and solving solutions, except that LEVF-P is slightly worse in improving solutions for the \esmall benchmark under the large radius. This is because the long-edge-vertex model needs fewer candidate points than the edge-vertex model for covering long edges, so that the LEVF-P formulation has fewer variables than the EF-P formulation.

\textbf{Comparison of formulations for the edge model.} As we have proposed several reformulations for \eqref{milpe}  of the edge model, we need to compare their strengths. We look at the big-M reformulation EF-P, the indicator reformulation EF-PI, and the disjunctive programming reformulation EF-PD. In addition, we also consider the   big-M formulations  strengthened by valid inequalities: EF-PV1 and EF-PV2. First, we observe that EF-PV2 is the worst, because it has a large number valid inequalities, and its  formulation is too large for \cplex. Second, EF-PI is always worse than EF-P, so the handling of indicator constraints by \cplex is not as good as our tailored big-M reformulation. Third, EF-PD and EF-PV1 are both better than EF-P for the \esmall benchmark but get worse for the \elarge benchmark. The performance of EF-PD  degrades more than EF-PV1 for large instances, because the number of additional variables and constraints of EF-PD is proportional to the number of nodes.

\textbf{Formulation selection.} Based on the aforementioned comparison and analysis, we can filter out useless formulations that are generally worse than the other formulations.  For problems in different scales,  we  need to  select appropriate formulations among the remaining ones (EF-P, EF-PD, EF-PV, and LEVF-P). The aforementioned comparison  shows that the sizes of formulations have a major impact on their performances. Thus, we select formulations according to the sizes of networks and radii. For each benchmark and radius, we use a performance profile to plot the number of instances versus the relative gap obtained by a formulation.  The formulation with a sharp increasing performance profile is considered to perform well. We show the profiles in \Cref{fig.pp_ss,fig.pp_sl,fig.pp_ls,fig.pp_ll}. From  \Cref{fig.pp_ss,fig.pp_sl}, we find that EF-PD and EF-PV1 perform well for the \esmall benchmark, EF-PD is better for small radii and EF-PV1 is better for large radii. From  \Cref{fig.pp_ls,fig.pp_ll}, we find that LEVF-P outperforms the others for the \elarge benchmark. In summary, there is no uniformly good formulation, and one needs to select formulations according to the sizes of problems. For very small problems (small networks and radii), we recommend using the EF-PD formulation; for large problems (large networks and radii), we recommend using the LEVF-P formulation; for the other problems, we recommend using the EF-PV1 formulation and keeping EF-P as an easy-to-implement fallback option.

 \begin{figure}[!h]
\centering{
\includegraphics[width = 0.8\columnwidth]{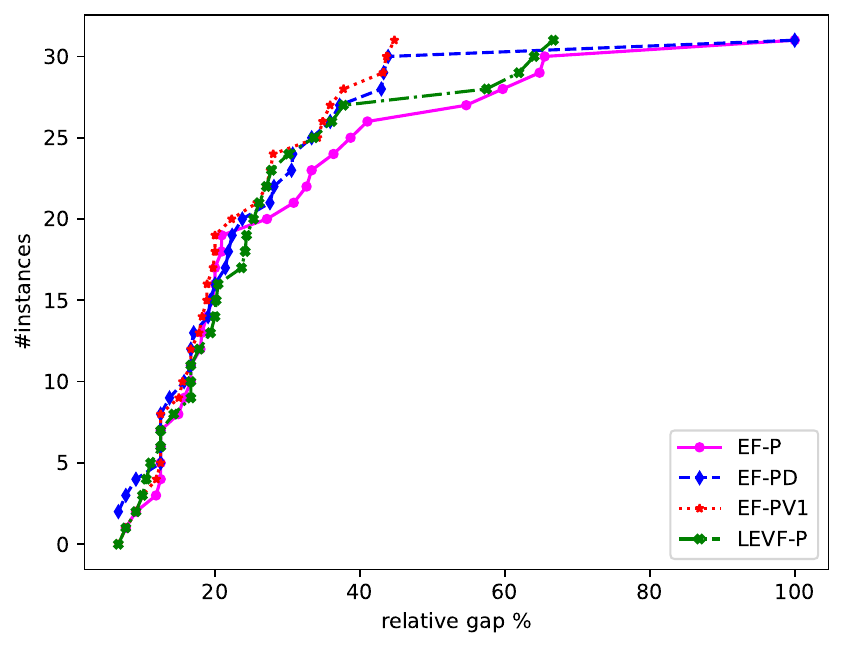}
}
\caption{Performance profile of four formulations for the \esmall benchmark and the Small radius}
\label{fig.pp_ss}
\end{figure}

 \begin{figure}[!h]
\centering{
\includegraphics[width = 0.8\columnwidth]{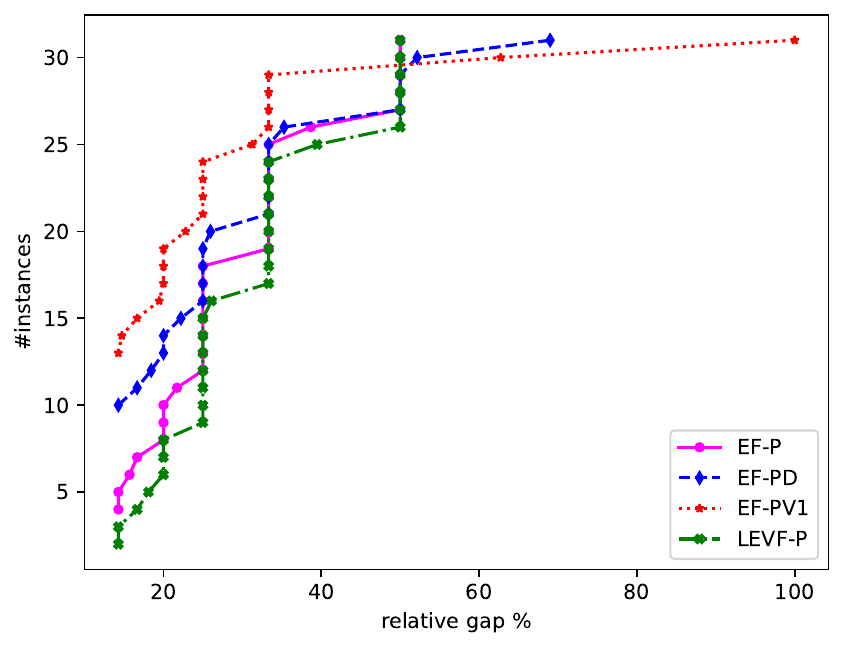}
}
\caption{Performance profile of four formulations for the \esmall benchmark and the Large radius}
\label{fig.pp_sl}
\end{figure}

 \begin{figure}[!h]
\centering{
\includegraphics[width = 0.7\columnwidth]{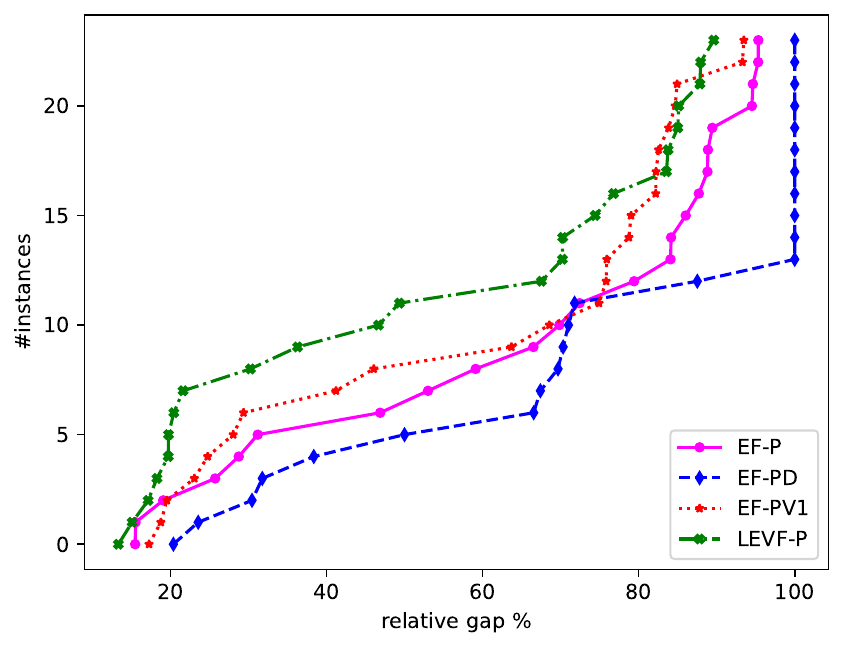}
}
\caption{Performance profile of four formulations for the \elarge benchmark and the Small radius}
\label{fig.pp_ls}
\end{figure}

 \begin{figure}[!h]
\centering{
\includegraphics[width = 0.7\columnwidth]{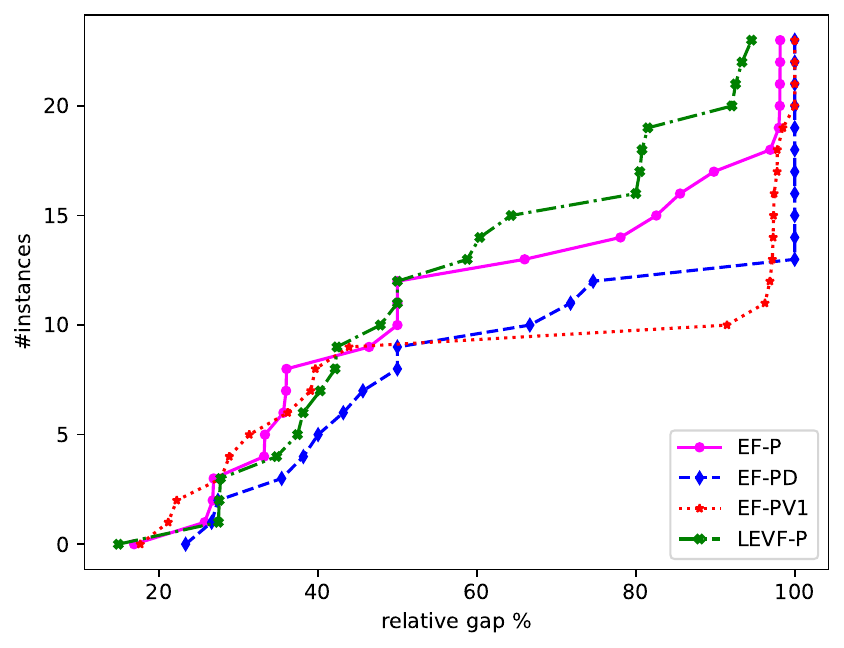}
}
\caption{Performance profile of four formulations for the \elarge benchmark and the Large radius}
\label{fig.pp_ll}
\end{figure}

\section{Conclusion}

In this paper, we introduce several novel MILP formulations and valid inequalities tailored for the edge model of the CSC problem. We investigate three strategies for formulating multivariate piecewise linear concave constraints that represent coverage conditions on edges, resulting in three distinct MILP formulations: one incorporating indicator constraints, another employing big-M techniques, and the third utilizing disjunctive programming techniques. We demonstrate that preprocessing techniques such as delimitation and bound tightening can be applied to enhance the efficiency of these formulations.

We provide a systematic categorization of both new and existing formulations, along with guidelines for selecting the most suitable formulation for specific problem instances. Our empirical findings suggest employing the EF-PD formulation for CSC problems involving small networks and radii, opting for the LEVF-P formulation for scenarios featuring large networks and radii, and utilizing the EF-PV1 formulation for other cases.

Looking ahead, the future evolution of this research could involve integrating tailored heuristics alongside exact methods to improve algorithmic performance. Our exploration of MILP formulations could serve as a valuable reference for developing exact approaches for addressing various network covering problems. It is anticipated that the size of MILP formulations will emerge as a critical factor influencing performance in these formulations.

\section{Data availibility}

All data analyzed during this study are publicly available in \href{https://github.com/lidingxu/cflg/}{https://github.com/lidingxu/cflg/}.

\section{Acknowledgments}
This publication was supported by the Chair ``Integrated Urban Mobility'', backed by L’X - \'Ecole Polytechnique and La Fondation de l’\'Ecole Polytechnique. The Partners of the Chair shall not under any circumstances accept any liability for the content of this publication, for which the author shall be solely liable.


\ifthenelse {\boolean{springer}}
{
\bibliographystyle{plain}
}
{
\bibliographystyle{plain}
}

\section*{Appendix}
\begin{landscape}
\scriptsize
\setlength{\tabcolsep}{1.9pt}
	\begin{longtable}{ll|rr|rrr|rrr|rrr|rrr|rrr|rrr|rrr|rrr|rrr|}
    	\caption{Detailed experimental results.
       	}
    	\label{table:root:detailed}\\
    \tabledefline{R}{radius, S (Small) or L (Large)}
     \tabledefline{SN}{splitted network}
 	\tabledefline{rt}{relative solving time (divided by the time limit 1800 seconds)}
	\tabledefline{$\sigma$}{relative dual gap}
	\tabledefline{$v_r$}{relative primal bound}
	\tabledefline{-}{the result is not available (the solver fails to load the MILP model)}
	\toprule
	\multirow{2}{*}{Instance}  & \multirow{2}{*}{R} &   \multicolumn{2}{c|}{SN} &   \multicolumn{3}{c|}{\texttt{EF}} & \multicolumn{3}{c|}{\texttt{EF-P}} & \multicolumn{3}{c|}{\texttt{EF-PI}} & \multicolumn{3}{c|}{\texttt{EF-PD}} &
	\multicolumn{3}{c|}{\texttt{EF-PV1}} & \multicolumn{3}{c|}{\texttt{EF-PV2}} &  \multicolumn{3}{c|}{\texttt{EVF-P}}  &
 \multicolumn{3}{c|}{\texttt{LEVF-P}}  \\
  & & nodes & edges  &  rt & $\sigma$  & $v_r$  & rt & $\sigma$  & $v_r$ & rt & $\sigma$  & $v_r$  & rt & $\sigma$  & $v_r$  & rt & $\sigma$  & $v_r$  & rt & $\sigma$  & $v_r$  & rt & $\sigma$  & $v_r$  & rt & $\sigma$  & $v_r$  \\
    	\toprule
 \multirow{2}{*}{\texttt{city265}}&S & 298 & 373.0&- & - & -&1.0 & 19.1\% & 34.9\%&1.0 & 18.8\% & 34.9\%&1.0 & 71.0\% & 96.6\%&1.0 & 18.8\% & 35.2\%&1.0 & 19.4\% & 35.2\%&1.0 & 28.6\% & 37.9\%&1.0 & 19.7\% & 35.9\%\\
&L & 178 & 253.0&- & - & -&1.0 & 26.8\% & 23.0\%&1.0 & 21.5\% & 21.9\%&1.0 & 26.6\% & 23.6\%&1.0 & 21.2\% & 21.9\%&1.0 & 23.6\% & 22.5\%&1.0 & 32.8\% & 23.6\%&1.0 & 27.6\% & 22.5\%\\
\multirow{2}{*}{\texttt{city628}}&S & 756 & 919.0&- & - & -&1.0 & 59.1\% & 69.7\%&1.0 & 60.5\% & 70.8\%&1.0 & 69.7\% & 89.4\%&1.0 & 19.5\% & 35.4\%&1.0 & 65.6\% & 83.1\%&1.0 & 39.4\% & 44.7\%&1.0 & 18.3\% & 34.7\%\\
&L & 486 & 649.0&- & - & -&1.0 & 26.9\% & 22.8\%&1.0 & 23.3\% & 21.8\%&1.0 & 71.8\% & 55.8\%&1.0 & 28.9\% & 22.6\%&0.8 & 30.0\% & 23.3\%&1.0 & 27.0\% & 22.2\%&1.0 & 27.5\% & 22.6\%\\
\multirow{2}{*}{\texttt{city479}}&S & 584 & 689.0&- & - & -&1.0 & 15.5\% & 34.2\%&1.0 & 48.1\% & 56.2\%&1.0 & 67.4\% & 89.7\%&1.0 & 74.9\% & 117.3\%&1.0 & 14.4\% & 34.2\%&1.0 & 41.5\% & 47.9\%&1.0 & 13.3\% & 33.9\%\\
&L & 379 & 484.0&- & - & -&1.0 & 16.9\% & 21.4\%&1.0 & 16.1\% & 21.4\%&1.0 & 23.4\% & 23.2\%&1.0 & 17.6\% & 21.6\%&0.9 & 19.8\% & 21.9\%&1.0 & 22.0\% & 22.2\%&1.0 & 14.9\% & 20.8\%\\
\multirow{2}{*}{\texttt{city132}}&S & 152 & 179.0&- & - & -&- & - & -&1.0 & 11.5\% & 36.8\%&- & - & -&1.0 & 11.9\% & 36.8\%&- & - & -&1.0 & 12.5\% & 36.8\%&1.0 & 10.5\% & 36.8\%\\
&L & 91 & 118.0&1.0 & 87.7\% & 129.7\%&1.0 & 15.7\% & 24.2\%&1.0 & 17.3\% & 24.2\%&1.0 & 18.4\% & 24.2\%&1.0 & 14.7\% & 24.2\%&1.0 & 14.3\% & 24.2\%&1.0 & 21.2\% & 24.2\%&1.0 & 18.1\% & 24.2\%\\
\multirow{2}{*}{\texttt{city213}}&S & 249 & 319.0&- & - & -&1.0 & 25.7\% & 32.5\%&1.0 & 25.6\% & 31.3\%&1.0 & 30.5\% & 34.9\%&1.0 & 24.8\% & 32.1\%&1.0 & 24.0\% & 32.1\%&1.0 & 30.0\% & 32.5\%&1.0 & 21.6\% & 30.9\%\\
&L & 136 & 206.0&- & - & -&1.0 & 36.0\% & 21.3\%&1.0 & 39.2\% & 22.1\%&1.0 & 38.2\% & 22.8\%&1.0 & 36.2\% & 22.8\%&1.0 & 36.5\% & 22.8\%&1.0 & 44.4\% & 22.8\%&1.0 & 40.3\% & 21.3\%\\
\multirow{2}{*}{\texttt{city276}}&S & 318 & 403.0&- & - & -&1.0 & 28.8\% & 35.8\%&1.0 & 45.3\% & 46.5\%&1.0 & 66.6\% & 78.3\%&1.0 & 23.1\% & 34.3\%&1.0 & 22.5\% & 33.3\%&1.0 & 50.4\% & 48.7\%&1.0 & 20.4\% & 33.0\%\\
&L & 186 & 271.0&- & - & -&1.0 & 36.0\% & 23.1\%&1.0 & 32.6\% & 22.0\%&1.0 & 40.0\% & 25.3\%&1.0 & 31.4\% & 22.0\%&1.0 & 34.3\% & 22.6\%&1.0 & 37.8\% & 22.6\%&1.0 & 37.4\% & 23.1\%\\
\multirow{2}{*}{\texttt{city268}}&S & 307 & 380.0&- & - & -&1.0 & 15.5\% & 34.2\%&1.0 & 18.9\% & 34.9\%&1.0 & 23.6\% & 36.8\%&1.0 & 17.3\% & 34.2\%&1.0 & 19.1\% & 35.8\%&1.0 & 38.0\% & 43.6\%&1.0 & 17.1\% & 34.2\%\\
&L & 188 & 261.0&- & - & -&1.0 & 33.2\% & 22.3\%&1.0 & 32.6\% & 22.9\%&1.0 & 27.4\% & 21.8\%&1.0 & 27.8\% & 21.8\%&0.9 & 26.7\% & 21.3\%&1.0 & 38.2\% & 22.9\%&1.0 & 34.8\% & 22.9\%\\
\multirow{2}{*}{\texttt{city771}}&S & 944 & 1123.0&- & - & -&1.0 & 66.5\% & 85.9\%&1.0 & 62.5\% & 74.6\%&1.0 & 70.3\% & 90.1\%&1.0 & 75.9\% & 118.3\%&1.0 & 76.1\% & 118.3\%&1.0 & 39.0\% & 44.5\%&1.0 & 15.1\% & 33.8\%\\
&L & 591 & 770.0&- & - & -&1.0 & 25.7\% & 21.0\%&1.0 & 25.4\% & 21.5\%&1.0 & 74.6\% & 59.9\%&1.0 & 22.2\% & 20.5\%&1.0 & 21.5\% & 20.1\%&1.0 & 30.1\% & 22.0\%&1.0 & 27.7\% & 21.3\%\\
\multirow{2}{*}{\texttt{city138}}&S & 162 & 201.0&- & - & -&1.0 & 18.9\% & 32.7\%&1.0 & 18.7\% & 32.1\%&1.0 & 21.9\% & 34.0\%&1.0 & 19.7\% & 33.3\%&1.0 & 18.9\% & 32.7\%&1.0 & 20.8\% & 32.7\%&1.0 & 17.8\% & 32.1\%\\
&L & 88 & 127.0&1.0 & 52.1\% & 35.2\%&1.0 & 21.7\% & 22.7\%&1.0 & 22.2\% & 22.7\%&1.0 & 22.2\% & 22.7\%&1.0 & 19.5\% & 22.7\%&1.0 & 20.6\% & 22.7\%&1.0 & 24.0\% & 22.7\%&1.0 & 26.1\% & 22.7\%\\
\multirow{2}{*}{\texttt{K100.3.con.red}}&S & 94 & 191.0&- & - & -&1.0 & 33.3\% & 25.5\%&1.0 & 40.2\% & 26.6\%&1.0 & 28.1\% & 26.6\%&1.0 & 34.9\% & 26.6\%&- & - & -&1.0 & 50.8\% & 27.7\%&1.0 & 27.1\% & 25.5\%\\
&L & 40 & 137.0&- & - & -&0.1 & 14.3\% & 17.5\%&0.2 & 14.3\% & 17.5\%&0.1 & 14.3\% & 17.5\%&0.2 & 14.3\% & 17.5\%&- & - & -&0.1 & 14.3\% & 17.5\%&0.0 & 14.3\% & 17.5\%\\
\multirow{2}{*}{\texttt{K100.9.red}}&S & 56 & 104.0&1.0 & 40.0\% & 35.7\%&1.0 & 15.7\% & 28.6\%&1.0 & 16.6\% & 28.6\%&0.9 & 6.7\% & 26.8\%&0.8 & 6.7\% & 26.8\%&0.5 & 21.9\% & 28.6\%&1.0 & 21.1\% & 28.6\%&0.6 & 6.7\% & 26.8\%\\
&L & 30 & 78.0&- & - & -&0.0 & 33.3\% & 10.0\%&0.0 & 33.3\% & 10.0\%&0.0 & 33.3\% & 10.0\%&0.0 & 0.0\% & 10.0\%&- & - & -&0.0 & 33.3\% & 10.0\%&0.0 & 33.3\% & 10.0\%\\
\multirow{2}{*}{\texttt{K100.6.red}}&S & 50 & 92.0&1.0 & 41.2\% & 34.0\%&1.0 & 19.6\% & 30.0\%&1.0 & 19.3\% & 30.0\%&1.0 & 19.4\% & 30.0\%&1.0 & 15.0\% & 30.0\%&0.8 & 25.6\% & 30.0\%&1.0 & 20.0\% & 30.0\%&1.0 & 19.4\% & 30.0\%\\
&L & 28 & 70.0&- & - & -&0.0 & 25.0\% & 14.3\%&0.0 & 25.0\% & 14.3\%&0.0 & 25.0\% & 14.3\%&0.0 & 25.0\% & 14.3\%&- & - & -&0.0 & 25.0\% & 14.3\%&0.1 & 25.0\% & 14.3\%\\
\multirow{2}{*}{\texttt{K100.2.red}}&S & 61 & 120.0&- & - & -&1.0 & 20.9\% & 29.5\%&1.0 & 23.3\% & 29.5\%&1.0 & 19.1\% & 29.5\%&1.0 & 25.8\% & 29.5\%&- & - & -&1.0 & 28.2\% & 29.5\%&1.0 & 20.4\% & 29.5\%\\
&L & 32 & 91.0&- & - & -&0.0 & 25.0\% & 12.5\%&0.1 & 25.0\% & 12.5\%&0.2 & 0.0\% & 12.5\%&0.0 & 0.0\% & 12.5\%&- & - & -&0.1 & 0.0\% & 12.5\%&0.0 & 25.0\% & 12.5\%\\
\multirow{2}{*}{\texttt{K100.10.red}}&S & 67 & 118.0&1.0 & 72.0\% & 74.6\%&1.0 & 14.9\% & 28.4\%&1.0 & 34.9\% & 29.9\%&1.0 & 22.4\% & 29.9\%&1.0 & 18.9\% & 28.4\%&0.4 & 27.5\% & 29.9\%&1.0 & 24.6\% & 29.9\%&1.0 & 20.2\% & 28.4\%\\
&L & 35 & 86.0&- & - & -&0.0 & 20.0\% & 14.3\%&0.1 & 20.0\% & 14.3\%&0.1 & 20.0\% & 14.3\%&0.1 & 20.0\% & 14.3\%&- & - & -&0.2 & 20.0\% & 14.3\%&0.2 & 20.0\% & 14.3\%\\
\multirow{2}{*}{\texttt{K100.1.red}}&S & 120 & 263.0&- & - & -&1.0 & 31.2\% & 32.5\%&1.0 & 36.2\% & 31.7\%&1.0 & 20.4\% & 33.3\%&1.0 & 28.0\% & 33.3\%&- & - & -&1.0 & 41.7\% & 33.3\%&1.0 & 19.8\% & 30.0\%\\
&L & 65 & 208.0&- & - & -&1.0 & 35.7\% & 15.4\%&1.0 & 29.9\% & 15.4\%&1.0 & 43.2\% & 16.9\%&1.0 & 43.9\% & 16.9\%&- & - & -&1.0 & 39.4\% & 15.4\%&1.0 & 42.4\% & 16.9\%\\
\multirow{2}{*}{\texttt{K100.5.con.red}}&S & 86 & 175.0&- & - & -&1.0 & 41.0\% & 29.1\%&1.0 & 40.2\% & 27.9\%&1.0 & 37.2\% & 30.2\%&1.0 & 35.9\% & 27.9\%&- & - & -&1.0 & 45.6\% & 29.1\%&1.0 & 33.6\% & 27.9\%\\
&L & 39 & 128.0&- & - & -&0.8 & 14.3\% & 17.9\%&0.3 & 14.3\% & 17.9\%&1.0 & 26.0\% & 20.5\%&1.0 & 31.3\% & 20.5\%&- & - & -&1.0 & 22.0\% & 17.9\%&1.0 & 14.3\% & 17.9\%\\
\multirow{2}{*}{\texttt{K100.7.red}}&S & 74 & 142.0&- & - & -&1.0 & 38.7\% & 31.1\%&1.0 & 40.7\% & 31.1\%&1.0 & 21.4\% & 31.1\%&1.0 & 22.3\% & 31.1\%&- & - & -&1.0 & 41.2\% & 31.1\%&1.0 & 23.7\% & 31.1\%\\
&L & 37 & 105.0&- & - & -&0.0 & 25.0\% & 10.8\%&0.0 & 25.0\% & 10.8\%&0.0 & 0.0\% & 10.8\%&0.0 & 0.0\% & 10.8\%&- & - & -&0.2 & 25.0\% & 10.8\%&0.1 & 25.0\% & 10.8\%\\
\multirow{2}{*}{\texttt{K100.4.con.red}}&S & 85 & 169.0&- & - & -&1.0 & 36.4\% & 27.1\%&1.0 & 35.7\% & 27.1\%&1.0 & 30.7\% & 29.4\%&1.0 & 28.0\% & 27.1\%&- & - & -&1.0 & 40.3\% & 28.2\%&1.0 & 25.3\% & 27.1\%\\
&L & 40 & 124.0&- & - & -&0.0 & 25.0\% & 10.0\%&0.2 & 20.0\% & 12.5\%&0.3 & 0.0\% & 10.0\%&0.0 & 0.0\% & 10.0\%&- & - & -&0.2 & 0.0\% & 10.0\%&0.4 & 25.0\% & 10.0\%\\
\multirow{2}{*}{\texttt{K100.8.con.red}}&S & 107 & 208.0&- & - & -&1.0 & 30.9\% & 29.9\%&1.0 & 31.3\% & 29.9\%&1.0 & 33.3\% & 31.8\%&1.0 & 34.2\% & 29.9\%&- & - & -&1.0 & 35.5\% & 29.9\%&1.0 & 27.8\% & 29.0\%\\
&L & 57 & 158.0&- & - & -&1.0 & 38.7\% & 19.3\%&1.0 & 23.6\% & 15.8\%&1.0 & 35.3\% & 17.5\%&1.0 & 22.8\% & 15.8\%&- & - & -&1.0 & 40.0\% & 17.5\%&1.0 & 39.5\% & 17.5\%\\
\multirow{2}{*}{\texttt{K100.con.red}}&S & 145 & 291.0&- & - & -&1.0 & 53.0\% & 27.6\%&1.0 & 53.5\% & 29.0\%&1.0 & 38.4\% & 24.8\%&1.0 & 41.2\% & 24.8\%&- & - & -&1.0 & 50.3\% & 24.1\%&1.0 & 46.7\% & 24.1\%\\
&L & 55 & 201.0&- & - & -&1.0 & 33.3\% & 16.4\%&1.0 & 36.1\% & 16.4\%&1.0 & 35.4\% & 16.4\%&1.0 & 39.1\% & 16.4\%&- & - & -&1.0 & 39.3\% & 16.4\%&1.0 & 38.1\% & 16.4\%\\
\multirow{2}{*}{\texttt{K400.7.con.red}}&S & 633 & 1275.0&- & - & -&1.0 & 88.9\% & 108.7\%&1.0 & 87.5\% & 109.0\%&- & - & -&1.0 & 82.2\% & 101.4\%&- & - & -&1.0 & 82.3\% & 65.9\%&1.0 & 83.8\% & 104.9\%\\
&L & 302 & 944.0&- & - & -&1.0 & 78.1\% & 31.1\%&1.0 & 97.8\% & 312.6\%&- & - & -&0.9 & 97.3\% & 312.6\%&- & - & -&1.0 & 98.4\% & 412.6\%&1.0 & 64.3\% & 20.9\%\\
    \endfirsthead  \\
\multirow{2}{*}{\texttt{K400.5.con.red}}&S & 600 & 1179.0&- & - & -&1.0 & 87.7\% & 106.7\%&1.0 & 88.0\% & 107.7\%&- & - & -&1.0 & 82.6\% & 108.3\%&- & - & -&1.0 & 81.3\% & 63.8\%&1.0 & 83.6\% & 105.0\%\\
&L & 295 & 874.0&- & - & -&1.0 & 82.6\% & 35.3\%&1.0 & 97.9\% & 296.3\%&- & - & -&1.0 & 91.5\% & 102.4\%&- & - & -&1.0 & 98.5\% & 396.3\%&1.0 & 80.5\% & 35.3\%\\
\multirow{2}{*}{\texttt{K400.9.con.red}}&S & 588 & 1239.0&- & - & -&1.0 & 88.8\% & 119.0\%&1.0 & 86.6\% & 103.2\%&- & - & -&1.0 & 68.6\% & 64.8\%&- & - & -&1.0 & 80.5\% & 64.1\%&1.0 & 70.3\% & 71.8\%\\
&L & 312 & 963.0&- & - & -&1.0 & 98.2\% & 308.7\%&1.0 & 98.2\% & 308.7\%&- & - & -&0.9 & 97.2\% & 308.7\%&- & - & -&1.0 & 98.7\% & 408.7\%&1.0 & 58.8\% & 20.5\%\\
\multirow{2}{*}{\texttt{K400.2.red}}&S & 709 & 1429.0&- & - & -&1.0 & 94.6\% & 201.6\%&1.0 & 94.6\% & 201.6\%&- & - & -&1.0 & 75.8\% & 70.7\%&- & - & -&1.0 & 83.2\% & 61.1\%&1.0 & 85.0\% & 109.0\%\\
&L & 311 & 1031.0&- & - & -&1.0 & 98.2\% & 331.5\%&1.0 & 98.2\% & 331.5\%&- & - & -&0.9 & 97.8\% & 331.5\%&- & - & -&1.0 & 98.7\% & 431.5\%&1.0 & 94.6\% & 112.5\%\\
\multirow{2}{*}{\texttt{K400.red}}&S & 715 & 1398.0&- & - & -&1.0 & 84.2\% & 67.4\%&1.0 & 89.0\% & 96.4\%&- & - & -&1.0 & 84.7\% & 93.7\%&- & - & -&1.0 & 85.2\% & 67.3\%&1.0 & 87.9\% & 114.7\%\\
&L & 296 & 979.0&- & - & -&1.0 & 98.0\% & 330.7\%&1.0 & 88.2\% & 55.7\%&- & - & -&0.9 & 97.3\% & 330.7\%&- & - & -&1.0 & 98.5\% & 430.7\%&1.0 & 80.8\% & 39.9\%\\
\multirow{2}{*}{\texttt{K400.1.con.red}}&S & 587 & 1224.0&- & - & -&1.0 & 84.1\% & 119.3\%&1.0 & 84.0\% & 112.4\%&1.0 & 87.5\% & 122.7\%&1.0 & 79.0\% & 96.8\%&- & - & -&1.0 & 75.9\% & 56.0\%&1.0 & 36.3\% & 34.9\%\\
&L & 314 & 951.0&- & - & -&1.0 & 96.9\% & 302.9\%&1.0 & 92.1\% & 114.6\%&- & - & -&0.9 & 96.3\% & 302.9\%&- & - & -&1.0 & 97.8\% & 402.9\%&1.0 & 42.2\% & 22.9\%\\
\multirow{2}{*}{\texttt{K400.8.con.red}}&S & 749 & 1501.0&- & - & -&1.0 & 95.3\% & 200.4\%&1.0 & 95.3\% & 200.4\%&- & - & -&1.0 & 93.5\% & 200.4\%&- & - & -&1.0 & 97.1\% & 300.4\%&1.0 & 89.6\% & 116.4\%\\
&L & 294 & 1046.0&- & - & -&1.0 & 85.6\% & 46.9\%&1.0 & 98.1\% & 355.8\%&- & - & -&- & - & -&- & - & -&1.0 & 98.6\% & 455.8\%&1.0 & 92.6\% & 95.2\%\\
\multirow{2}{*}{\texttt{K400.6.con.red}}&S & 789 & 1583.0&- & - & -&1.0 & 95.3\% & 200.6\%&1.0 & 95.2\% & 200.6\%&- & - & -&1.0 & 93.3\% & 200.6\%&- & - & -&1.0 & 97.2\% & 300.6\%&1.0 & 87.9\% & 115.6\%\\
&L & 321 & 1115.0&- & - & -&1.0 & 98.1\% & 347.4\%&1.0 & 98.1\% & 347.4\%&- & - & -&0.9 & 98.5\% & 347.4\%&- & - & -&1.0 & 98.6\% & 447.4\%&1.0 & 92.1\% & 98.8\%\\
\multirow{2}{*}{\texttt{K400.4.red}}&S & 516 & 1103.0&- & - & -&1.0 & 86.0\% & 105.0\%&1.0 & 85.7\% & 109.3\%&- & - & -&1.0 & 78.8\% & 94.0\%&- & - & -&1.0 & 77.8\% & 55.4\%&1.0 & 49.3\% & 38.8\%\\
&L & 270 & 857.0&- & - & -&1.0 & 98.1\% & 317.4\%&1.0 & 98.1\% & 317.4\%&- & - & -&0.9 & 97.4\% & 317.4\%&- & - & -&1.0 & 98.6\% & 417.4\%&1.0 & 93.4\% & 107.8\%\\
\multirow{2}{*}{\texttt{K400.10.con.red}}&S & 671 & 1373.0&- & - & -&1.0 & 94.5\% & 204.6\%&1.0 & 94.5\% & 204.6\%&- & - & -&1.0 & 84.9\% & 113.4\%&- & - & -&1.0 & 96.5\% & 304.6\%&1.0 & 74.4\% & 64.2\%\\
&L & 300 & 1002.0&- & - & -&1.0 & 89.8\% & 52.7\%&1.0 & 98.4\% & 334.0\%&- & - & -&0.9 & 97.8\% & 334.0\%&- & - & -&1.0 & 98.8\% & 434.0\%&1.0 & 81.5\% & 32.0\%\\
\multirow{2}{*}{\texttt{K400.3.con.red}}&S & 595 & 1191.0&- & - & -&1.0 & 89.4\% & 110.6\%&1.0 & 89.7\% & 112.6\%&- & - & -&1.0 & 82.2\% & 88.9\%&- & - & -&1.0 & 83.4\% & 63.9\%&1.0 & 85.1\% & 104.5\%\\
&L & 273 & 869.0&- & - & -&1.0 & 66.0\% & 21.6\%&1.0 & 97.7\% & 318.3\%&- & - & -&0.9 & 96.9\% & 318.3\%&- & - & -&1.0 & 98.4\% & 418.3\%&1.0 & 60.4\% & 23.1\%\\
\multirow{2}{*}{\texttt{K200.con.red}}&S & 184 & 374.0&- & - & -&1.0 & 46.9\% & 37.5\%&1.0 & 42.2\% & 34.8\%&1.0 & 31.8\% & 36.4\%&1.0 & 29.4\% & 34.8\%&- & - & -&1.0 & 53.8\% & 46.2\%&1.0 & 30.3\% & 32.6\%\\
&L & 109 & 299.0&- & - & -&1.0 & 46.4\% & 13.8\%&1.0 & 42.9\% & 12.8\%&1.0 & 45.7\% & 14.7\%&1.0 & 39.7\% & 12.8\%&- & - & -&1.0 & 48.9\% & 13.8\%&1.0 & 47.9\% & 13.8\%\\
\multirow{2}{*}{\texttt{r100.312}}&S & 15 & 17.0&0.0 & 16.7\% & 40.0\%&0.0 & 16.7\% & 40.0\%&0.0 & 16.7\% & 40.0\%&0.0 & 16.7\% & 40.0\%&0.0 & 16.7\% & 40.0\%&0.0 & 16.7\% & 40.0\%&0.0 & 16.7\% & 40.0\%&0.0 & 16.7\% & 40.0\%\\
&L & 10 & 12.0&0.0 & 33.3\% & 30.0\%&0.0 & 33.3\% & 30.0\%&0.0 & 33.3\% & 30.0\%&0.0 & 33.3\% & 30.0\%&0.0 & 0.0\% & 20.0\%&0.0 & 0.0\% & 20.0\%&0.0 & 33.3\% & 30.0\%&0.0 & 33.3\% & 30.0\%\\
\multirow{2}{*}{\texttt{r150.325}}&S & 26 & 36.0&0.2 & 12.5\% & 30.8\%&0.0 & 12.5\% & 30.8\%&0.0 & 12.5\% & 30.8\%&0.0 & 12.5\% & 30.8\%&0.0 & 12.5\% & 30.8\%&0.0 & 12.5\% & 30.8\%&0.1 & 12.5\% & 30.8\%&0.0 & 12.5\% & 30.8\%\\
&L & 15 & 25.0&0.0 & 33.3\% & 20.0\%&0.0 & 33.3\% & 20.0\%&0.0 & 33.3\% & 20.0\%&0.0 & 33.3\% & 20.0\%&0.0 & 33.3\% & 20.0\%&0.0 & 33.3\% & 20.0\%&0.0 & 33.3\% & 20.0\%&0.0 & 33.3\% & 20.0\%\\
\multirow{2}{*}{\texttt{r200.234}}&S & 39 & 53.0&1.0 & 24.5\% & 28.2\%&0.1 & 9.1\% & 28.2\%&0.1 & 9.1\% & 28.2\%&0.2 & 9.1\% & 28.2\%&0.0 & 9.1\% & 28.2\%&0.1 & 9.1\% & 28.2\%&0.6 & 9.1\% & 28.2\%&0.3 & 9.1\% & 28.2\%\\
&L & 20 & 34.0&0.0 & 33.3\% & 15.0\%&0.0 & 33.3\% & 15.0\%&0.0 & 33.3\% & 15.0\%&0.0 & 33.3\% & 15.0\%&0.0 & 33.3\% & 15.0\%&0.0 & 33.3\% & 15.0\%&0.0 & 33.3\% & 15.0\%&0.0 & 33.3\% & 15.0\%\\
\multirow{2}{*}{\texttt{r150.445}}&S & 40 & 70.0&1.0 & 51.5\% & 27.5\%&0.9 & 12.5\% & 20.0\%&1.0 & 22.3\% & 20.0\%&0.5 & 12.5\% & 20.0\%&1.0 & 15.5\% & 20.0\%&0.9 & 21.3\% & 20.0\%&1.0 & 34.4\% & 22.5\%&1.0 & 26.1\% & 20.0\%\\
&L & 15 & 45.0&0.1 & 33.3\% & 20.0\%&0.0 & 50.0\% & 13.3\%&0.0 & 50.0\% & 13.3\%&0.0 & 50.0\% & 13.3\%&0.0 & 0.0\% & 13.3\%&0.0 & 0.0\% & 13.3\%&0.0 & 50.0\% & 13.3\%&0.0 & 50.0\% & 13.3\%\\
\multirow{2}{*}{\texttt{r100.413}}&S & 14 & 17.0&0.0 & 16.7\% & 42.9\%&0.0 & 16.7\% & 42.9\%&0.0 & 16.7\% & 42.9\%&0.0 & 16.7\% & 42.9\%&0.0 & 16.7\% & 42.9\%&0.0 & 16.7\% & 42.9\%&0.0 & 16.7\% & 42.9\%&0.0 & 16.7\% & 42.9\%\\
&L & 10 & 13.0&0.0 & 0.0\% & 20.0\%&0.0 & 0.0\% & 20.0\%&0.0 & 33.3\% & 30.0\%&0.0 & 0.0\% & 20.0\%&0.0 & 0.0\% & 20.0\%&0.0 & 0.0\% & 20.0\%&0.0 & 0.0\% & 20.0\%&0.0 & 33.3\% & 30.0\%\\
\multirow{2}{*}{\texttt{r150.114}}&S & 19 & 18.0&0.0 & 12.5\% & 42.1\%&0.0 & 12.5\% & 42.1\%&0.0 & 12.5\% & 42.1\%&0.0 & 12.5\% & 42.1\%&0.0 & 12.5\% & 42.1\%&0.0 & 12.5\% & 42.1\%&0.0 & 12.5\% & 42.1\%&0.0 & 11.1\% & 47.4\%\\
&L & 15 & 14.0&0.0 & 25.0\% & 26.7\%&0.0 & 25.0\% & 26.7\%&0.0 & 0.0\% & 26.7\%&0.0 & 25.0\% & 26.7\%&0.0 & 25.0\% & 26.7\%&0.0 & 25.0\% & 26.7\%&0.0 & 25.0\% & 26.7\%&0.0 & 25.0\% & 26.7\%\\
\multirow{2}{*}{\texttt{r200.349}}&S & 46 & 75.0&1.0 & 33.3\% & 26.1\%&1.0 & 18.2\% & 23.9\%&1.0 & 18.2\% & 23.9\%&1.0 & 15.7\% & 23.9\%&1.0 & 17.8\% & 23.9\%&0.9 & 19.9\% & 23.9\%&1.0 & 24.6\% & 23.9\%&1.0 & 24.2\% & 23.9\%\\
&L & 20 & 49.0&0.8 & 33.3\% & 15.0\%&0.0 & 50.0\% & 10.0\%&0.0 & 50.0\% & 10.0\%&0.0 & 0.0\% & 10.0\%&0.0 & 0.0\% & 10.0\%&0.0 & 0.0\% & 10.0\%&0.0 & 50.0\% & 10.0\%&0.0 & 50.0\% & 10.0\%\\
\multirow{2}{*}{\texttt{r100.29}}&S & 12 & 11.0&0.0 & 20.0\% & 41.7\%&0.0 & 20.0\% & 41.7\%&0.0 & 0.0\% & 41.7\%&0.0 & 0.0\% & 41.7\%&0.0 & 20.0\% & 41.7\%&0.0 & 20.0\% & 41.7\%&0.0 & 20.0\% & 41.7\%&0.0 & 16.7\% & 50.0\%\\
&L & 10 & 9.0&0.0 & 0.0\% & 20.0\%&0.0 & 33.3\% & 30.0\%&0.0 & 33.3\% & 30.0\%&0.0 & 33.3\% & 30.0\%&0.0 & 0.0\% & 20.0\%&0.0 & 0.0\% & 20.0\%&0.0 & 0.0\% & 20.0\%&0.0 & 33.3\% & 30.0\%\\
\multirow{2}{*}{\texttt{r100.112}}&S & 16 & 18.0&0.0 & 0.0\% & 31.3\%&0.0 & 20.0\% & 31.2\%&0.0 & 20.0\% & 31.2\%&0.0 & 20.0\% & 31.2\%&0.0 & 20.0\% & 31.2\%&0.0 & 20.0\% & 31.2\%&0.0 & 20.0\% & 31.2\%&0.0 & 20.0\% & 31.2\%\\
&L & 10 & 12.0&0.0 & 33.3\% & 30.0\%&0.0 & 0.0\% & 20.0\%&0.0 & 0.0\% & 20.0\%&0.0 & 0.0\% & 20.0\%&0.0 & 33.3\% & 30.0\%&0.0 & 33.3\% & 30.0\%&0.0 & 33.3\% & 30.0\%&0.0 & 33.3\% & 30.0\%\\
\multirow{2}{*}{\texttt{r200.123}}&S & 29 & 32.0&0.1 & 10.0\% & 34.5\%&0.0 & 0.0\% & 34.5\%&0.0 & 10.0\% & 34.5\%&0.0 & 0.0\% & 34.5\%&0.0 & 10.0\% & 34.5\%&0.0 & 10.0\% & 34.5\%&0.0 & 9.1\% & 37.9\%&0.0 & 10.0\% & 34.5\%\\
&L & 20 & 23.0&0.0 & 25.0\% & 20.0\%&0.0 & 20.0\% & 25.0\%&0.0 & 20.0\% & 25.0\%&0.0 & 0.0\% & 20.0\%&0.0 & 20.0\% & 25.0\%&0.0 & 20.0\% & 25.0\%&0.0 & 20.0\% & 25.0\%&0.0 & 20.0\% & 25.0\%\\
\multirow{2}{*}{\texttt{r150.222}}&S & 25 & 32.0&0.0 & 12.5\% & 32.0\%&0.0 & 12.5\% & 32.0\%&0.0 & 12.5\% & 32.0\%&0.0 & 12.5\% & 32.0\%&0.0 & 12.5\% & 32.0\%&0.0 & 12.5\% & 32.0\%&0.0 & 12.5\% & 32.0\%&0.0 & 12.5\% & 32.0\%\\
&L & 15 & 22.0&0.0 & 0.0\% & 20.0\%&0.0 & 33.3\% & 20.0\%&0.0 & 33.3\% & 20.0\%&0.0 & 0.0\% & 20.0\%&0.0 & 33.3\% & 20.0\%&0.0 & 33.3\% & 20.0\%&0.0 & 33.3\% & 20.0\%&0.0 & 33.3\% & 20.0\%\\
\multirow{2}{*}{\texttt{r200.469}}&S & 52 & 101.0&1.0 & 52.9\% & 32.7\%&1.0 & 32.6\% & 23.1\%&1.0 & 30.9\% & 23.1\%&1.0 & 27.6\% & 23.1\%&1.0 & 27.1\% & 23.1\%&0.6 & 30.2\% & 23.1\%&1.0 & 38.8\% & 23.1\%&1.0 & 37.7\% & 23.1\%\\
&L & 20 & 69.0&1.0 & 77.9\% & 45.0\%&0.0 & 50.0\% & 10.0\%&0.0 & 50.0\% & 10.0\%&0.0 & 50.0\% & 10.0\%&0.0 & 0.0\% & 10.0\%&- & - & -&0.0 & 50.0\% & 10.0\%&0.0 & 50.0\% & 10.0\%\\
\multirow{2}{*}{\texttt{r250.4112}}&S & 83 & 170.0&- & - & -&1.0 & 64.8\% & 21.7\%&1.0 & 54.2\% & 18.1\%&1.0 & 43.3\% & 19.3\%&1.0 & 43.2\% & 18.1\%&- & - & -&1.0 & 65.3\% & 20.5\%&1.0 & 66.7\% & 21.7\%\\
&L & 25 & 112.0&- & - & -&0.0 & 50.0\% & 8.0\%&0.0 & 50.0\% & 8.0\%&0.2 & 50.0\% & 8.0\%&0.2 & 0.0\% & 8.0\%&- & - & -&0.0 & 50.0\% & 8.0\%&0.0 & 0.0\% & 8.0\%\\
\multirow{2}{*}{\texttt{r300.3131}}&S & 92 & 193.0&- & - & -&1.0 & 65.5\% & 25.0\%&1.0 & 58.8\% & 21.7\%&1.0 & 43.0\% & 21.7\%&1.0 & 44.8\% & 21.7\%&- & - & -&1.0 & 66.9\% & 23.9\%&1.0 & 64.0\% & 21.7\%\\
&L & 30 & 131.0&- & - & -&0.5 & 33.3\% & 10.0\%&0.9 & 33.3\% & 10.0\%&1.0 & 52.1\% & 10.0\%&1.0 & 62.7\% & 10.0\%&- & - & -&0.5 & 33.3\% & 10.0\%&1.0 & 50.0\% & 10.0\%\\
\multirow{2}{*}{\texttt{r300.275}}&S & 67 & 112.0&1.0 & 47.8\% & 34.3\%&1.0 & 27.2\% & 25.4\%&1.0 & 29.0\% & 25.4\%&1.0 & 30.6\% & 26.9\%&1.0 & 27.6\% & 25.4\%&1.0 & 32.9\% & 26.9\%&1.0 & 30.2\% & 25.4\%&1.0 & 36.1\% & 26.9\%\\
&L & 30 & 75.0&1.0 & 77.0\% & 33.3\%&0.4 & 25.0\% & 13.3\%&0.2 & 25.0\% & 13.3\%&0.9 & 25.0\% & 13.3\%&0.1 & 25.0\% & 13.3\%&0.2 & 44.9\% & 13.3\%&0.7 & 25.0\% & 13.3\%&0.5 & 25.0\% & 13.3\%\\
\multirow{2}{*}{\texttt{r250.136}}&S & 40 & 51.0&1.0 & 27.5\% & 35.0\%&0.0 & 7.7\% & 32.5\%&0.1 & 7.7\% & 32.5\%&0.1 & 7.7\% & 32.5\%&0.0 & 7.7\% & 32.5\%&0.0 & 7.7\% & 32.5\%&0.2 & 7.7\% & 32.5\%&0.0 & 7.7\% & 32.5\%\\
&L & 25 & 36.0&0.0 & 20.0\% & 20.0\%&0.0 & 0.0\% & 16.0\%&0.0 & 0.0\% & 16.0\%&0.0 & 0.0\% & 16.0\%&0.0 & 0.0\% & 16.0\%&0.0 & 0.0\% & 16.0\%&0.0 & 0.0\% & 16.0\%&0.0 & 0.0\% & 16.0\%\\
\multirow{2}{*}{\texttt{r400.184}}&S & 78 & 122.0&- & - & -&1.0 & 20.9\% & 26.9\%&1.0 & 21.6\% & 26.9\%&1.0 & 23.8\% & 28.2\%&1.0 & 18.2\% & 26.9\%&1.0 & 24.0\% & 26.9\%&1.0 & 30.5\% & 29.5\%&1.0 & 24.3\% & 26.9\%\\
&L & 40 & 84.0&1.0 & 50.0\% & 20.0\%&0.0 & 16.7\% & 15.0\%&0.1 & 16.7\% & 15.0\%&0.3 & 16.7\% & 15.0\%&0.0 & 16.7\% & 15.0\%&- & - & -&0.0 & 16.7\% & 15.0\%&0.1 & 16.7\% & 15.0\%\\
\multirow{2}{*}{\texttt{r400.3219}}&S & 149 & 328.0&- & - & -&1.0 & 69.8\% & 20.8\%&1.0 & 72.0\% & 22.8\%&1.0 & 71.8\% & 28.2\%&1.0 & 63.7\% & 24.2\%&- & - & -&1.0 & 74.6\% & 21.5\%&1.0 & 70.2\% & 20.1\%\\
&L & 40 & 219.0&- & - & -&0.6 & 50.0\% & 5.0\%&1.0 & 66.7\% & 7.5\%&0.2 & 50.0\% & 5.0\%&- & - & -&- & - & -&1.0 & 66.7\% & 7.5\%&1.0 & 80.0\% & 12.5\%\\
\multirow{2}{*}{\texttt{r300.4188}}&S & 117 & 275.0&- & - & -&1.0 & 72.4\% & 19.7\%&1.0 & 69.4\% & 17.9\%&1.0 & 50.0\% & 20.5\%&1.0 & 46.1\% & 17.9\%&- & - & -&1.0 & 76.6\% & 23.1\%&1.0 & 67.5\% & 18.8\%\\
&L & 30 & 188.0&- & - & -&0.2 & 50.0\% & 6.7\%&0.1 & 50.0\% & 6.7\%&0.5 & 50.0\% & 6.7\%&- & - & -&- & - & -&0.1 & 50.0\% & 6.7\%&0.2 & 50.0\% & 6.7\%\\
\multirow{2}{*}{\texttt{r300.154}}&S & 52 & 76.0&1.0 & 29.4\% & 32.7\%&1.0 & 11.9\% & 30.8\%&1.0 & 12.9\% & 30.8\%&1.0 & 13.7\% & 30.8\%&1.0 & 12.5\% & 30.8\%&1.0 & 15.0\% & 30.8\%&1.0 & 16.7\% & 30.8\%&1.0 & 14.3\% & 30.8\%\\
&L & 30 & 54.0&1.0 & 36.7\% & 16.7\%&0.0 & 20.0\% & 16.7\%&0.0 & 20.0\% & 16.7\%&0.1 & 20.0\% & 16.7\%&0.0 & 20.0\% & 16.7\%&0.0 & 20.0\% & 16.7\%&0.0 & 20.0\% & 16.7\%&0.1 & 20.0\% & 16.7\%\\
\multirow{2}{*}{\texttt{r400.2148}}&S & 112 & 220.0&- & - & -&1.0 & 59.7\% & 25.9\%&1.0 & 55.4\% & 24.1\%&1.0 & 43.9\% & 24.1\%&1.0 & 43.6\% & 23.2\%&- & - & -&1.0 & 59.9\% & 24.1\%&1.0 & 61.9\% & 25.9\%\\
&L & 40 & 148.0&- & - & -&1.0 & 50.0\% & 10.0\%&1.0 & 50.0\% & 10.0\%&1.0 & 69.0\% & 12.5\%&- & - & -&- & - & -&1.0 & 60.0\% & 12.5\%&1.0 & 50.0\% & 10.0\%\\
\multirow{2}{*}{\texttt{r250.258}}&S & 53 & 86.0&1.0 & 47.1\% & 32.1\%&1.0 & 18.0\% & 24.5\%&1.0 & 20.3\% & 24.5\%&1.0 & 17.1\% & 24.5\%&1.0 & 18.9\% & 24.5\%&1.0 & 19.1\% & 24.5\%&1.0 & 23.8\% & 24.5\%&1.0 & 30.1\% & 26.4\%\\
&L & 25 & 58.0&1.0 & 45.5\% & 16.0\%&0.1 & 25.0\% & 16.0\%&0.1 & 25.0\% & 16.0\%&0.2 & 25.0\% & 16.0\%&0.0 & 25.0\% & 16.0\%&0.1 & 25.0\% & 16.0\%&0.2 & 25.0\% & 16.0\%&0.3 & 25.0\% & 16.0\%\\
\multirow{2}{*}{\texttt{r400.4297}}&S & 191 & 448.0&- & - & -&1.0 & 79.4\% & 18.3\%&1.0 & 79.4\% & 17.8\%&- & - & -&0.9 & 83.8\% & 44.0\%&- & - & -&1.0 & 80.8\% & 19.9\%&1.0 & 76.8\% & 16.8\%\\
&L & 40 & 297.0&- & - & -&0.8 & 50.0\% & 5.0\%&1.0 & 66.7\% & 7.5\%&1.0 & 66.7\% & 7.5\%&- & - & -&- & - & -&1.0 & 66.7\% & 7.5\%&0.7 & 50.0\% & 5.0\%\\
\multirow{2}{*}{\texttt{r250.398}}&S & 74 & 147.0&1.0 & 93.5\% & 198.6\%&1.0 & 54.7\% & 23.0\%&1.0 & 50.0\% & 21.6\%&1.0 & 35.9\% & 21.6\%&1.0 & 37.7\% & 21.6\%&0.4 & 40.0\% & 21.6\%&1.0 & 56.1\% & 21.6\%&1.0 & 57.5\% & 21.6\%\\
&L & 25 & 98.0&- & - & -&0.2 & 0.0\% & 8.0\%&0.0 & 50.0\% & 8.0\%&0.7 & 0.0\% & 8.0\%&0.9 & 0.0\% & 8.0\%&- & - & -&0.1 & 50.0\% & 8.0\%&0.0 & 50.0\% & 8.0\%\\    
\bottomrule
\end{longtable}
\end{landscape}

\bibliography{reference}

\end{document}